\documentclass[12pt,reqno]{amsart}
\usepackage[top=1in, bottom=1in, left=1in, right=1in]{geometry}
\usepackage{times, amsthm, amssymb, amsmath, amsfonts,
  amsthm,mathrsfs, amsrefs, bm}
\usepackage{hyperref}
\usepackage[usenames,dvipsnames]{color}
\usepackage{enumitem}
\usepackage{tikz}
\usepackage{subcaption}
\usepackage{caption}
\usepackage{tikz-3dplot}
\usepackage{xcolor}

\newtheorem{theorem}{Theorem}[section]
\newtheorem{proposition}[theorem]{Proposition}
\newtheorem{corollary}[theorem]{Corollary}

\newtheorem{lemma}[theorem]{Lemma}

\theoremstyle{definition}
\newtheorem{definition}[theorem]{Definition}

\newtheorem{example}[theorem]{Example}
\newtheorem{notation}[theorem]{Notation}
\newtheorem{remark}[theorem]{Remark}

\newcommand{\NN}{\mathbb{N}}
\newcommand{\ZZ}{\mathbb{Z}}
\newcommand{\RR}{\mathbb{R}}

\newcommand{\field}{\Bbbk}

\newcommand{\pf}{\mathfrak{p}_F}
\newcommand{\pg}{\mathfrak{p}_G}

\newcommand{\<}{\langle}
\renewcommand{\>}{\rangle}
\newcommand{\minus}{\smallsetminus}

\DeclareMathOperator{\mult}{mult}
\DeclareMathOperator{\stdMono}{std}
\DeclareMathOperator{\stdPairs}{stdPairs}

\makeatletter
\@namedef{subjclassname@2020}{
  \textup{2020} Mathematics Subject Classification}
\makeatother

\colorlet{idealcolor}{black}
\colorlet{stdmcolor}{blue!60}
\colorlet{idealregioncolor}{gray!20}
\newcommand{\twofigure}{0.498\linewidth}
\newcommand{\fourfigure}{0.497\linewidth}
\newcommand{\threefigure}{0.329\linewidth}
\newcommand{\threedtwofigure}{0.497\linewidth}

\begin{document}

\title{Standard pairs for monomial ideals in semigroup rings}
\author[Laura Felicia Matusevich]{Laura Felicia Matusevich}
\address[Laura Felicia Matusevich]{Department of Mathematics \\
Texas A\&M University \\ College Station, TX 77843.}
\email[Laura Felicia Matusevich]{laura@math.tamu.edu}
\author[Byeongsu Yu]{Byeongsu Yu}
\address[Byeongsu Yu]{Department of Mathematics \\
Texas A\&M University \\ College Station, TX 77843.}
\email[Byeongsu Yu]{byeongsu.yu@math.tamu.edu}

\begin{abstract}
We extend the notion of standard pairs to the context of monomial
ideals in semigroup rings. Standard pairs can be used as a data structure
to encode such monomial ideals, providing an alternative to
generating sets that is well suited to computing
intersections, decompositions, and multiplicities. We give algorithms to
compute standard pairs from generating sets and vice versa and make all of our
results effective. 
We assume that the underlying semigroup ring is positively graded, but
not necessarily normal.
The lack of normality is at the root of most challenges, subtleties, and innovations in this work.
\end{abstract}


\maketitle

\section{Introduction}
\label{sec:intro}

The polynomial ring on $d$ variables over a field is
$\ZZ^{d}$-graded ring, where the degree of a monomial is defined to be
its exponent 
vector. This is a \emph{fine grading}: every graded piece is a vector space
over the base field of dimension at most one. From this point of view,
a monomial ideal is a $\ZZ^{d}$-homogeneous ideal. Affine semigroup
rings are also finely graded, and it makes sense to talk about
monomial ideals in this context as well. Monomial ideals in polynomial
rings are a mainstay of combinatorial commutative algebra, and have
been extensively studied (see for instance, the
texts~\cites{MR1453579,MR2110098}). In contrast, much less is known
about 
monomial ideals in affine semigroup rings (but see~\cites{MR1481087, MR1892318,
  MR2168288}). This is not surprising, as general semigroup rings 
do not satisfy many properties the polynomial ring enjoys.

A monomial ideal $I$ is determined by the monomials that
belong to $I$, but also by the monomials that do \emph{not} belong to
$I$, which are known as the \emph{standard monomials} of $I$. Usually, we
encode a monomial ideal through a (finite) monomial generating
set, a description that is best suited to working with the monomials in
$I$. Standard pairs, introduced in~\cite{MR1339920}, give a
finite way of encoding the standard monomials of a monomial ideal in a
polynomial ring. If $I$ and $J$ are monomial ideals, the set of
standard monomials of $I\cap J$ is the union of the standard monomials
of $I$ and $J$. This makes standard pairs particularly well-adapted to
tasks involving intersections and decompositions of monomial ideals.
The main goal of this article is to extend this point of view to
monomial ideals in semigroup rings.

Beyond their original use in~\cite{MR1339920} to give combinatorial
bounds for degrees of projective schemes, standard pairs have been used
to prove properties of initial ideals of toric
ideals~\cites{MR1682705, MR1993049, MR2142838}, in applications related to
optimization~\cite{MR1700541}, in combinatorial settings~\cite{MR2046082}
and to compute series solutions of hypergeometric
systems~\cite{MR1734566}. An algorithm for computing standard pairs is
given in~\cites{MR1734566, MR1949544}, and is
implemented in the computer algebra system \texttt{Macaulay2}~\cite{M2}.

The standard pair definition given in~\cite{MR1339920} naturally extends to monomial
ideals in semigroup rings. However, even in the normal case,
standard pairs in this context exhibit behavior that is not present
over the polynomial ring. Nevertheless, basic results about standard
pairs still hold (with different proofs): A monomial ideal $I$
has finitely many standard 
pairs (Theorem~\ref{thm:finitelyManyStdP}). The associated primes of $I$ can be read
off immediately from its standard pairs, and the
standard pairs of $I$ can be used to give combinatorial 
primary and irreducible decompositions of $I$
(Theorem~\ref{thm:primaryDecomposition} and
Proposition~\ref{prop:irreducibleDecomposition}). Finally, counting 
(equivalence classes) of standard pairs yields multiplicities of
associated primes (Proposition~\ref{prop:multiplicity}). 

We are particularly concerned with the computational aspects of standard
pairs. This is motivated by the difficulty of computing 
combinatorial structures associated with general binomial ideals (ideals
generated by polynomials with at most two terms). Since an
affine semigroup ring is isomorphic to the quotient of a polynomial
ring modulo a prime binomial ideal, the quotient of an affine
semigroup ring by a monomial ideal is isomorphic to the quotient of a
polynomial ring modulo the sum of a prime binomial ideal and a
monomial ideal. In other words, monomial ideals in semigroup rings can be identified
with special kinds of binomial ideals.

The general study of binomial ideals was initiated
in~\cite{MR1394747}, where it was shown that binomial 
ideals can have a primary decomposition consisting of binomial ideals (when the base
field is algebraically closed). Specialized algorithms for finding
such decompositions can be found in~\cite{MR1394747} (see
also~\cites{MR1784749, MR2652314, MR3556446}). Combinatorial structures
controlling the decompositions of binomial ideals were given
in~\cites{MR2593293, MR3267140}, but there are currently no known
algorithms to compute these structures, even if a primary decomposition is known by
other means. 

Standard pairs represent a different combinatorial approach to
decompositions of our special of binomial ideals: they are not
the specialization of the structures from~\cites{MR2593293, MR3267140}. Moreover, one of
our main results gives a method to compute the standard pairs of a
monomial ideal from a generating set
(Theorem~\ref{thm:computeStdPairs}). We also provide a method to
compute a generating set of a monomial ideal given its standard pairs
(Theorem~\ref{thm:pairsToGens}). 
Using standard pairs, we
describe a method to produce 
an irredundant irreducible decomposition of a monomial ideal in a
semigroup ring
(Theorem~\ref{thm:computeIrreducibleDecomposition}). An irreducible decomposition algorithm 
already existed in the case that the underlying semigroup ring is
normal~\cite{MR2168288}. However, the general case given in
Theorem~\ref{thm:computeIrreducibleDecomposition} was indicated as an
open problem in the notes of~\cite {MR2110098}*{Chapter~11}. Finally,
computing intersections of monomial ideals using standard pairs is
particularly straightforward (Remark~\ref{rmk:computeIntersections}).

As with most computations involving affine semigroup rings, our
procedures for finding and using the
standard pairs of a monomial ideal use ideas and techniques from
convex discrete optimization. Loosely speaking, our algorithms require
solving multiple integer linear programs (ILPs), which are famously known as NP-complete. However, our specific situation is not as bad as it
sounds: if we fix the ambient affine semigroup ring, then finding
standard pairs is not general integer programming, but integer
programming \emph{in fixed dimension}, which is famously solvable in
polynomial time~\cite{MR727410}. Even further, 
when the ambient ring is fixed, all the ILPs we need to solve arise
from a single known matrix (the matrix of 
generators of the affine semigroup) in a finite number of possible
ways. This means that, after some possibly costly pre-computations
(that can be done once and stored), finding and using standard pairs in
a fixed semigroup ring should not be too computationally intensive. Implementation of these
ideas is done by the \texttt{StdPairs}~\cite{Yu22} package on \texttt{SageMath} developed by the second author. Some more details on this work in progress can be found in Section~\ref{sec:implementation}.


\subsection*{Outline}
This article is organized as follows.
In Section~\ref{sec:preliminaries} we set notation and review
background material. Section~\ref{sec:stdPairs} develops the  
theory of standard pairs and shows how to use standard pairs to give
combinatorial descriptions of primary and irreducible decompositions
of monomial ideals in semigroup rings. In Section~\ref{sec:algorithms}
we describe algorithms to compute and use standard pairs. Algorithms
outlined in this section include: computation of standard pairs given the
generators of a monomial ideal, computation of the generators of a
monomial ideal given its standard pairs, computation of 
irreducible (and primary) decompositions, computation of multiplicities. 
Section~\ref{sec:implementation} explains the authors' project to
implement these methods in the computer algebra system \texttt{SageMath} and \texttt{Macaulay2}.

\subsection*{Acknowledgments}
We are grateful to Sarah Witherspoon, Abraham Mart\'{\i}n del Campo
Sanchez, Gabriela Jeronimo, Jack Jeffries, Siddharth Mahendraker,
Claudiu Raicu, Alexander Yong for conversations we had while working
on this project, especially at the \emph{XXIII Coloquio
Latinoamericano de Algebra} in Mexico City in August 2019, and at
\emph{Algebra, Geometry and Combinatorics Day - XVIIII} in
at the University of Illinois Urbana-Champaign in March 2020.

\section{Preliminaries}
\label{sec:preliminaries}

We adopt the convention that $\NN =\{0,1,2,\dots\}$ is the set of
nonnegative integers and $\field$ is an infinite field.

Throughout this article, $A = \{a_1,\dots,a_n\} \subset \ZZ^d \minus\{0\}$
is a fixed finite set of nonzero lattice points, called a
\emph{configuration}. We may abuse  
notation and use $A$ to also denote the $d\times n$ integer matrix
whose columns are $a_1,\dots,a_n$. We let $\NN A$ be the monoid of
nonnegative integer combinations of $a_1,\dots,a_n$; in the most
commonly used terminology, $\NN A$ is called an \emph{affine
  semigroup}. Similarly $\ZZ A$ is the (free abelian) group of integer
combinations of elements of $A$, and $\RR_{\geq 0} A$ is the cone of
nonnegative real combinations of elements of $A$. To simplify
notation, we assume that $\ZZ A =\ZZ^d$.  (When $\ZZ A\neq
\ZZ^d$ we use $\ZZ A$ as a ground lattice, and our proofs go
through essentially unchanged.)

We also assume that  $\RR_{\geq 0} A$ is \emph{strongly convex cone},
meaning that it contains 
no lines. We point out that
Lemma~\ref{lemma:stdPairMinimality}, which is used in the proof of
Theorem~\ref{thm:finitelyManyStdP}, fails without strong convexity.

We say that $F \subseteq A$ is a \emph{face} of $A$ if $\RR_{\geq 0} F$ is a face of
$\RR_{\geq 0} A$, and $\RR_{\geq 0} F \cap A = F$. In this case, we
also abuse notation and use $F$ to denote both a configuration and
its corresponding matrix (whose columns are the elements of the
configuration). It is known that $\RR_{\geq 0} A$ is strongly convex if and
only if $\{0\}$ is a face of $\RR_{\geq 0} A$. We use the convention that
$F = \varnothing$ refers to the origin as a face of $A$.

\begin{definition}
\label{def:integerDistance}
If $H$ is a \emph{facet} (a codimension one face) of $A$, we define
its \emph{primitive integral support function} $\varphi_H:\RR^d \to
\RR$ by the following properties:
\begin{enumerate}
\item $\varphi_H$ is linear,
\item $\varphi_H(\ZZ^d) = \ZZ$,
\item $\varphi_H(a_i) \geq 0$ for $i=1,\dots,n$,
\item $\varphi_H(a_i) = 0$ if and only if $a_i \in H$.
\end{enumerate}
\end{definition}
Primitive integral support functions give a measure of how far a point
is from a facet of $A$: if $a \in \ZZ^d$, $\varphi_H(a)$ is the number
of hyperplanes parallel to $\RR H$ that pass through integer points,
and lie between $a$ and $\RR H$, with a sign to indicate whether $a$
is on the side of $\RR H$ that contains $\RR_{\geq 0} A$.

We denote by $\NN F$ the affine semigroup generated by a face
$F$, $\RR_{\geq 0} F$ the cone over $F$, and $\RR F$ the real linear span
of $F$. Since $\RR_{\geq 0} F \cap A = F$, we have that $\NN F = \NN A
\cap \RR_{\geq 0} F$.

We work with the semigroup ring $\field[\NN A] =
\field[t^{a_1},\dots,t^{a_n}]$, which is a subring of the Laurent
polynomial ring $\field[t^{\pm}] = \field[t_1^{\pm 1},\dots,t_d^{\pm
  1}]$. This ring has the presentation 
$\field[\NN A] \cong \field[x]/I_A$, where
$\field[x]=\field[x_1,\dots,x_n]$, and $I_A=\< x^u -x^v \mid u,v \in
\NN^n, Au=Av\>$ (here we have used $A$ to denote a matrix). Since
$\RR_{\geq 0} A$ is strongly convex, the only multiplicative units in
$\field[\NN A]$ are the nonzero elements of the field $\field$.

The semigroup $\NN A$ is \emph{saturated} if $\RR_{\geq 0} A \cap \ZZ
A = \NN A$. When $\NN A$ is not saturated, $(\RR_{\geq 0} A \cap \ZZ
A)\minus \NN A$ is called the set of \emph{holes} of $\NN A$. It is a well-known result that $\NN A$ is saturated if and only if the domain $\field[\NN A]$
is normal (meaning that it is integrally closed over its field
of fractions).

The ring $\field[\NN A]$ is $\ZZ A = \ZZ^d$-graded, via $\deg(t^a) =
a$. In the presentation $\field[\NN A] \cong \field[x]/I_A$, the
grading is induced by setting $\deg(x_i)=a_i$. Strong convexity 
$\RR_{\geq 0} A$ means that this is a \emph{positive} grading: the
unique maximal $\ZZ^d$-homogeneous ideal of $\field[\NN A]$ is
$\<t^{a_1},\dots,t^{a_n}\>$.
A $\ZZ^d$-homogeneous ideal in $\field[\NN A]$ is called a \emph{monomial ideal}. Equivalently, a
monomial ideal in $\field[\NN A] \subset \field[t^\pm]$ is an ideal generated by Laurent
monomials.

There is a one to one inclusion reversing correspondence between the
set of monomial prime ideals in $\field[\NN A]$ and the set of faces of $\RR_{\geq 0} A$, given in the following statement.

\begin{lemma}[\cite{MR2110098}*{Lemma~7.10}]
\label{lemma:monomialPrimes}
If $F$ is a face of $A$, the monomial ideal 
\[
\pf = \< t^a \mid a \in \NN A \minus \NN F \>  \subset
\field[\NN A]
\]
is prime. All prime monomial ideals in $\field[\NN A]$ are of this form.
\end{lemma}

\begin{notation}
\label{notation:divisibility}
We emphasize that, throughout this article, divisibility refers to the
ring $\field[\NN A]$, and not to $\field[t^\pm]$. To be completely
precise, $t^{a'} \mid t^a$ means that $a-a' \in \NN A$. We abuse
terminology, and also state that $a'$ divides $a$ in this case.
\end{notation}

Affine semigroup rings are Noetherian, as they are quotients of
polynomial rings. This can be restated as a version of Dickson's Lemma.

\begin{lemma}
  \label{lemma:Dickson}
  Let $S$ be a nonempty subset of $\NN A$ such that no two elements of
  $S$ are comparable with respect to divisibility. Then $S$ is finite.
\end{lemma}

\begin{proof}
  By contradiction, assume that $S$ contains an infinite sequence
  $\{b_i\}_{i=1}^\infty$. Consider $I_j
  =\<t^{b_1},\dots,t^{b_j}\>$ for $j\geq 1$. Since $t^b \in I_j$ if and only if
  $t^{b_i} | t^b$ for some $1 \leq i \leq j$, we see that $I_1
  \subsetneq I_2 \subsetneq I_3 \subsetneq \cdots$ is an infinite
  ascending chain, which contradicts Noetherianity of $\field[\NN A]$. 
\end{proof}

We close out this section by providing some running examples. We represent
semigroup rings and monomial ideals pictorially by plotting exponent
vectors of monomials. In Figures~\ref{fig:2dafffinesemigp}
and~\ref{fig:3dafffinesemigp}, the (exponents of) standard monomials
of the given ideal are colored blue, while the (exponents of) monomials
in the ideal are colored black. 

\begin{example}
  \label{ex:ideals}
\
\begin{enumerate}[leftmargin=*]
\item  \label{item:2dpoly}
Let $A$ be a $d\times d$ identity matrix. Then $\NN A=\NN^{d}$ and
consequently $\field[\NN A] \cong \field[x_{1},\cdots, x_{d}]$. A face
of $\NN A$ is a set of all nonnegative integral combinations of a
subset of (the columns of) $A$. In Figure~\ref{fig:2dpolyring}, the
shaded region represents the monomial ideal $I= \langle
x^3y, xy^2 \rangle \subset \field[x,y]$.
\item \label{item:2dnormal}
Let $A = \big[\begin{smallmatrix} 1 & 1 & 1 \\ 0 & 1 &
    2 \end{smallmatrix}\big]$. Then $\NN A$ is a saturated
  semigroup, and $\field[\NN A]\cong \field[x, xy, xy^2]$ is a normal
  semigroup ring, a subring of $\field[x,y]$.
  Figure~\ref{fig:2dnormal} illustrates the ideal $\langle
x^2y^2, x^3y \rangle \subseteq \field[\NN A]$.
\item \label{item:3dnormal}
Let $A = \left[\begin{smallmatrix}0 & 1 & 0 & 1 \\ 0 & 0 & 1 & 1 \\
    1 & 1 & 1 & 1 \end{smallmatrix}\right]$. In this case, $\field[\NN A] \cong
\field[z,xz,yz,xyz]$ is a saturated affine semigroup ring. We depict the ideal
$\langle x^2 z^2, x^2yz^2, x^2yz^2 \rangle \subset \field[\NN A]$ in
Figure~\ref{fig:3dnormal}.
\item \label{item:3dnonnormal}
Let $A=\left[\begin{smallmatrix} 0 & 0 & 1 & 1 & 1 & 1 \\ 2 & 0 & 0 &
    1 & 0 & 1 \\ 0 & 2 & 0 & 0 & 1 & 1\end{smallmatrix}\right]$. Then
$\field[\NN A]\cong \field[x,xy,xz,xyz,y^2,z^2]$ is a subring of
$\field[x,y,z]$. In this case, $\NN A$ is not saturated, and the set
of holes is $\{ (a,b,0) \mid (a,b) \in \NN^2 \minus (2\NN \times 2\NN) \}$.
Figure~\ref{fig:3dnonnormal} depicts the ideal $\langle
x,xyz,xyz^{2} \rangle\subset \field[\NN A]$. Holes are represented by
white circles.
\end{enumerate}
\end{example}

\begin{figure*}[t!]
    \centering
    \begin{subfigure}[t]{\twofigure}
        \centering
        \begin{tikzpicture}[scale=\linewidth/14cm]
	\fill[idealregioncolor] (1,6) -- (1,2) -- (3,2) -- (3,1) -- (6,1) -- (6,6) -- cycle ;
	  \draw[step=1cm,gray,very thin] (-1,-1) grid (6,6);
	  \draw [<->] (0,7) node (yaxis) [above] {$y$}
       	 |- (7,0) node (xaxis) [right] {$x$};
	\draw (-1.5,0) -- (0,0);
	\draw (0,-1.5) -- (0,0);

	\foreach \x in {0,...,6}
	\foreach \y in {0,...,6}
	\draw[black,fill=stdmcolor] (\x,\y) circle (3 pt);

	\foreach \x in {3,...,6}
	\foreach \y in {1,...,6}
	\draw[black,fill=idealcolor] (\x,\y) circle (3 pt);

	\foreach \x in {1,...,6}
	\foreach \y in {2,...,6}
	\draw[black,fill=idealcolor] (\x,\y) circle (3 pt);
	\end{tikzpicture}
	 \caption{$\langle x^3y, xy^2 \rangle$ (shaded region) in $\field[x,y]$}
	 \label{fig:2dpolyring}
    \end{subfigure}
    ~ 
    \begin{subfigure}[t]{\twofigure}
        \centering
        \begin{tikzpicture}[scale=\linewidth/14cm]
	\fill[idealregioncolor] (12,2) -- (16,2) -- (16,6) -- (14,6) -- cycle ;
	\fill[idealregioncolor] (13,1) -- (16,1) -- (16,6) -- (15.5,6) -- cycle ;

	  \draw[step=1cm,gray,very thin] (9,-1) grid (16,6);
	  \draw [<->] (10,7) node (yaxis) [above] {$y$}
       	 |- (17,0) node (xaxis) [right] {$x$};
	\draw (8.5,0) -- (10,0);
	\draw (10,-1.5) -- (10,0);
	\draw[black,<->](13.5,7) -- (10,0) -- (16,0)  ;

	\foreach \x in {10,...,16}
	\draw[black,fill=stdmcolor] (\x,0) circle (3 pt);
	\foreach \x in {10,...,13}
	{
		\draw[black,fill=stdmcolor] (\x,2*\x-20) circle (3 pt);
	};
	\foreach \x in {10,...,12}
	{
		\draw[black,fill=stdmcolor] (\x+1,1) circle (3 pt);
		\draw[black,fill=stdmcolor] (\x+1,2*\x-20+1) circle (3 pt);
	};
	\foreach \x in {10,...,14}
	\draw[black,fill=idealcolor] (2+\x,2) circle (3 pt);
	\foreach \x in {10,...,13}
	\draw[black,fill=idealcolor] (3+\x,1) circle (3 pt);
	\foreach \x in {13,...,16}
	\foreach \y in {3,4}
	\draw[black,fill=idealcolor] (\x,\y) circle (3 pt);
	\foreach \x in {14,...,16}
	\foreach \y in {5,6}
	\draw[black,fill=idealcolor] (\x,\y) circle (3 pt);
	\end{tikzpicture}
        \caption{$\langle x^2y^2, x^3y \rangle$ (shaded region) in $\field[x,xy,xy^2]$}
        \label{fig:2dnormal}
    \end{subfigure}
    \caption{Examples of ideals in two-dimensional affine semigroup rings}
    \label{fig:2dafffinesemigp}
\end{figure*}
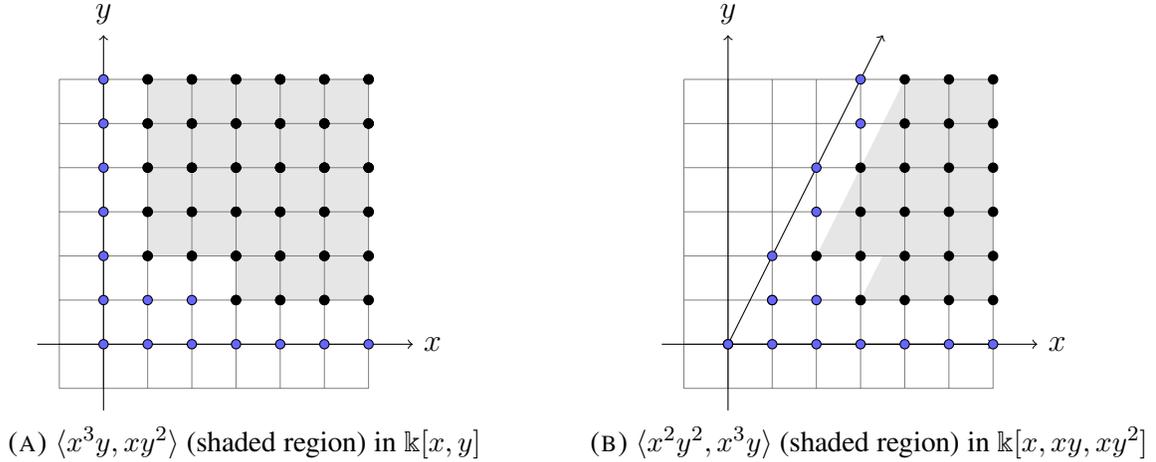

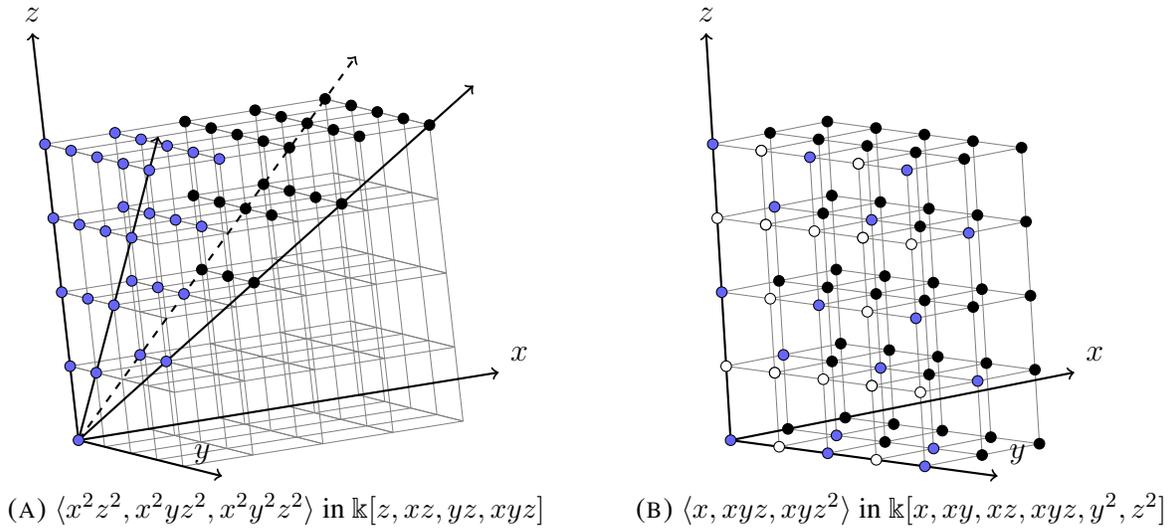
\begin{figure*}[t!]
    \centering
    \begin{subfigure}[t]{\twofigure}
        \centering
	 \tdplotsetmaincoords{90}{90} 
        \begin{tikzpicture}[tdplot_main_coords, scale=1]
	\tdplotsetrotatedcoords{40}{-10}{30}
	\foreach \x in {0,...,4}
	{
		\draw[step=1cm,gray,very thin, tdplot_rotated_coords] (\x,0,0) -- (\x,4,0);
		\draw[step=1cm,gray,very thin, tdplot_rotated_coords] (0,\x,0) -- (4,\x,0);
		\draw[step=1cm,gray,very thin, tdplot_rotated_coords] (\x,0,1) -- (\x,4,1);
		\draw[step=1cm,gray,very thin, tdplot_rotated_coords] (0,\x,1) -- (4,\x,1);
		\draw[step=1cm,gray,very thin, tdplot_rotated_coords] (\x,0,2) -- (\x,4,2);
		\draw[step=1cm,gray,very thin, tdplot_rotated_coords] (0,\x,2) -- (4,\x,2);
		\draw[step=1cm,gray,very thin, tdplot_rotated_coords] (\x,0,3) -- (\x,4,3);
		\draw[step=1cm,gray,very thin, tdplot_rotated_coords] (0,\x,3) -- (4,\x,3);
		\draw[step=1cm,gray,very thin, tdplot_rotated_coords] (\x,0,4) -- (\x,4,4);
		\draw[step=1cm,gray,very thin, tdplot_rotated_coords] (0,\x,4) -- (4,\x,4);

		\draw[step=1cm,gray,very thin, tdplot_rotated_coords] (0,\x,0) -- (0,\x,4);
		\draw[step=1cm,gray,very thin, tdplot_rotated_coords] (0,0,\x) -- (0,4,\x);
		\draw[step=1cm,gray,very thin, tdplot_rotated_coords] (1,\x,0) -- (1,\x,4);
		\draw[step=1cm,gray,very thin, tdplot_rotated_coords] (1,0,\x) -- (1,4,\x);
		\draw[step=1cm,gray,very thin, tdplot_rotated_coords] (2,\x,0) -- (2,\x,4);
		\draw[step=1cm,gray,very thin, tdplot_rotated_coords] (2,0,\x) -- (2,4,\x);
		\draw[step=1cm,gray,very thin, tdplot_rotated_coords] (3,\x,0) -- (3,\x,4);
		\draw[step=1cm,gray,very thin, tdplot_rotated_coords] (3,0,\x) -- (3,4,\x);
		\draw[step=1cm,gray,very thin, tdplot_rotated_coords] (4,\x,0) -- (4,\x,4);
		\draw[step=1cm,gray,very thin, tdplot_rotated_coords] (4,0,\x) -- (4,4,\x);
	};
	
	\draw[thick,->,tdplot_rotated_coords] (0,0,0) -- (6,0,0) node[anchor=south west]{$x$}; 
	\draw[thick,->,tdplot_rotated_coords] (0,0,0) -- (0,5.5,0) node[anchor=south east]{$y$}; 
	\draw[thick,->,tdplot_rotated_coords] (0,0,0) -- (0,0,5.5) node[anchor=south]{$z$};
	\draw[thick,->,tdplot_rotated_coords] (0,0,0) -- (0,4.5,4.5);
	\draw[thick,dashed,->,tdplot_rotated_coords] (0,0,0) -- (4.5,0,4.5);
	\draw[thick, ->,tdplot_rotated_coords] (0,0,0) -- (4.5,4.5,4.5);

	\foreach \y in {0,...,4}
	\draw[black,fill=stdmcolor,tdplot_rotated_coords] (0,0,\y) circle (2 pt);
	\draw[black,fill=stdmcolor,tdplot_rotated_coords] (0,1,1) circle (2 pt);
	\draw[black,fill=stdmcolor,tdplot_rotated_coords] (0,1,2) circle (2 pt);
	\draw[black,fill=stdmcolor,tdplot_rotated_coords] (0,2,2) circle (2 pt);
	\draw[black,fill=stdmcolor,tdplot_rotated_coords] (0,1,3) circle (2 pt);
	\draw[black,fill=stdmcolor,tdplot_rotated_coords] (0,2,3) circle (2 pt);
	\draw[black,fill=stdmcolor,tdplot_rotated_coords] (0,3,3) circle (2 pt);
	\draw[black,fill=stdmcolor,tdplot_rotated_coords] (0,1,4) circle (2 pt);
	\draw[black,fill=stdmcolor,tdplot_rotated_coords] (0,2,4) circle (2 pt);
	\draw[black,fill=stdmcolor,tdplot_rotated_coords] (0,3,4) circle (2 pt);
	\draw[black,fill=stdmcolor,tdplot_rotated_coords] (0,4,4) circle (2 pt);

	\foreach \y in {0,...,3}
	\draw[black,fill=stdmcolor,tdplot_rotated_coords] (1,0,1+\y) circle (2 pt);

	\foreach \y in {0,...,3}
	\draw[black,fill=stdmcolor,tdplot_rotated_coords] (1,1+\y,1+\y) circle (2 pt);

	\draw[black,fill=stdmcolor,tdplot_rotated_coords] (1,1,2) circle (2 pt);
	\draw[black,fill=stdmcolor,tdplot_rotated_coords] (1,1,3) circle (2 pt);
	\draw[black,fill=stdmcolor,tdplot_rotated_coords] (1,2,3) circle (2 pt);
	\draw[black,fill=stdmcolor,tdplot_rotated_coords] (1,1,4) circle (2 pt);
	\draw[black,fill=stdmcolor,tdplot_rotated_coords] (1,2,4) circle (2 pt);
	\draw[black,fill=stdmcolor,tdplot_rotated_coords] (1,3,4) circle (2 pt);

	\draw[black,fill=idealcolor,tdplot_rotated_coords] (2,0,2) circle (2 pt);
	\draw[black,fill=idealcolor,tdplot_rotated_coords] (2,1,2) circle (2 pt);
	\draw[black,fill=idealcolor,tdplot_rotated_coords] (2,2,2) circle (2 pt);
	\draw[black,fill=idealcolor,tdplot_rotated_coords] (2,0,3) circle (2 pt);
	\draw[black,fill=idealcolor,tdplot_rotated_coords] (2,1,3) circle (2 pt);
	\draw[black,fill=idealcolor,tdplot_rotated_coords] (2,2,3) circle (2 pt);
	\draw[black,fill=idealcolor,tdplot_rotated_coords] (2,3,3) circle (2 pt);
	\draw[black,fill=idealcolor,tdplot_rotated_coords] (3,0,3) circle (2 pt);
	\draw[black,fill=idealcolor,tdplot_rotated_coords] (3,1,3) circle (2 pt);
	\draw[black,fill=idealcolor,tdplot_rotated_coords] (3,2,3) circle (2 pt);
	\draw[black,fill=idealcolor,tdplot_rotated_coords] (3,3,3) circle (2 pt);
	\draw[black,fill=idealcolor,tdplot_rotated_coords] (2,0,4) circle (2 pt);
	\draw[black,fill=idealcolor,tdplot_rotated_coords] (2,1,4) circle (2 pt);
	\draw[black,fill=idealcolor,tdplot_rotated_coords] (2,2,4) circle (2 pt);
	\draw[black,fill=idealcolor,tdplot_rotated_coords] (2,3,4) circle (2 pt);
	\draw[black,fill=idealcolor,tdplot_rotated_coords] (2,4,4) circle (2 pt);
	\draw[black,fill=idealcolor,tdplot_rotated_coords] (3,0,4) circle (2 pt);
	\draw[black,fill=idealcolor,tdplot_rotated_coords] (3,1,4) circle (2 pt);
	\draw[black,fill=idealcolor,tdplot_rotated_coords] (3,2,4) circle (2 pt);
	\draw[black,fill=idealcolor,tdplot_rotated_coords] (3,3,4) circle (2 pt);
	\draw[black,fill=idealcolor,tdplot_rotated_coords] (3,4,4) circle (2 pt);
	\draw[black,fill=idealcolor,tdplot_rotated_coords] (4,0,4) circle (2 pt);
	\draw[black,fill=idealcolor,tdplot_rotated_coords] (4,1,4) circle (2 pt);
	\draw[black,fill=idealcolor,tdplot_rotated_coords] (4,2,4) circle (2 pt);
	\draw[black,fill=idealcolor,tdplot_rotated_coords] (4,3,4) circle (2 pt);
	\draw[black,fill=idealcolor,tdplot_rotated_coords] (4,4,4) circle (2 pt);

	\end{tikzpicture}
	 \caption{$\langle x^2 z^2, x^2yz^2, x^2y^{2}z^2 \rangle$ in $\field[z,xz,yz,xyz]$}
	 \label{fig:3dnormal}
    \end{subfigure}
    ~ 
    \begin{subfigure}[t]{\twofigure}
        \centering
       \tdplotsetmaincoords{90}{90} 
	\begin{tikzpicture}[tdplot_main_coords, scale=1]
	\tdplotsetrotatedcoords{20}{-10}{30}
	\foreach \x in {0,...,4}
	{
		\draw[step=1cm,gray,very thin, tdplot_rotated_coords] (0,\x,0) -- (0,\x,4);
		\draw[step=1cm,gray,very thin, tdplot_rotated_coords] (0,0,\x) -- (0,4,\x);
		\draw[step=1cm,gray,very thin, tdplot_rotated_coords] (1,\x,0) -- (1,\x,4);
		\draw[step=1cm,gray,very thin, tdplot_rotated_coords] (1,0,\x) -- (1,4,\x);
		\draw[step=1cm,gray,very thin, tdplot_rotated_coords] (2,\x,0) -- (2,\x,4);
		\draw[step=1cm,gray,very thin, tdplot_rotated_coords] (2,0,\x) -- (2,4,\x);
	};
	\foreach \y in {0,...,4}
	\foreach \z in {0,...,4}
	\draw[step=1cm,gray,very thin, tdplot_rotated_coords] (0,\y,\z) -- (2,\y,\z);

	\draw[thick,->,tdplot_rotated_coords] (0,0,0) -- (6,0,0) node[anchor=south west]{$x$}; 
	\draw[thick,->,tdplot_rotated_coords] (0,0,0) -- (0,5.5,0) node[anchor=south west]{$y$}; 
	\draw[thick,->,tdplot_rotated_coords] (0,0,0) -- (0,0,5.5) node[anchor=south]{$z$};

	\foreach \y in {0,...,2}
	\foreach \z in {0,...,2}
	{
		\draw[black,fill=idealcolor,tdplot_rotated_coords] (1,2*\y,2*\z) circle (2 pt);
	};
	\foreach \y in {0,...,1}
	\foreach \z in {0,...,1}
	\draw[black,fill=idealcolor,tdplot_rotated_coords] (1,1+2*\y,1+2*\z) circle (2 pt);

	\foreach \y in {0,...,1}
	\foreach \z in {0,...,1}
	\draw[black,fill=idealcolor,tdplot_rotated_coords] (1,1+2*\y,2+2*\z) circle (2 pt);
	\foreach \y in {0,...,4}
	\foreach \z in {0,...,4}
	\draw[black,fill=idealcolor,tdplot_rotated_coords] (2,\y,\z) circle (2 pt);

	\foreach \y in {0,...,1}
	\foreach \z in {0,...,1}
	{
		\draw[black,fill=white,tdplot_rotated_coords] (0,2*\y+1,2*\z) circle (2 pt);
		\draw[black,fill=white,tdplot_rotated_coords] (0,2*\y+1,1) circle (2 pt);
		\draw[black,fill=white,tdplot_rotated_coords] (0,2*\y,2*\z+1) circle (2 pt);
		\draw[black,fill=white,tdplot_rotated_coords] (0,2*\y+1,3) circle (2 pt);
		\draw[black,fill=white,tdplot_rotated_coords] (0,2*\y+1,4) circle (2 pt);
	};
	\draw[black,fill=white,tdplot_rotated_coords] (0,4,1) circle (2 pt);
	\draw[black,fill=white,tdplot_rotated_coords] (0,4,3) circle (2 pt);

	\foreach \z in {0,...,2}
	\foreach \y in {0,...,2}
	{
		\draw[black,fill=stdmcolor,tdplot_rotated_coords] (0,2*\y,2*\z) circle (2 pt);
	};
	\foreach \z in {0,...,1}
	\foreach \y in {0,...,2}
	{
		\draw[black,fill=stdmcolor,tdplot_rotated_coords] (1,2*\y,1+2*\z) circle (2 pt);
	};
	\foreach \y in {0,...,1}
	{
		\draw[black,fill=stdmcolor,tdplot_rotated_coords] (1,1+2*\y,0) circle (2 pt);
	};

	\end{tikzpicture}
        \caption{$\langle x,xyz,xyz^{2} \rangle$ in $\field[x,xy,xz,xyz,y^2,z^2]$}
        \label{fig:3dnonnormal}
    \end{subfigure}
    \caption{Examples of ideals in three-dimensional affine semigroup rings}
    \label{fig:3dafffinesemigp}
\end{figure*}

\section{Standard Pairs, decompositions, and multiplicities}
\label{sec:stdPairs}

In this section we develop the theory of standard pairs in the context
of monomial ideals in affine semigroup rings. We then use standard pairs to
describe primary and irreducible decompositions of a monomial ideal,
and to compute multiplicities of associated primes. Our first step is
to introduce the general combinatorial framework for this section.

\begin{definition}
  \label{def:Pairs}
  We call $(a,F)$, where $a\in \NN A$ and $F$ is a face of $A$, a pair of $A$. In this case, $(a,F)$ \emph{belongs} to the face $F$.
  \begin{enumerate}[leftmargin=*]
    \item \label{item:prec}
 $(a,F) \prec (b,G)$ denotes the containment $a+\NN F \subset b+ \NN G$. As the notation implies, this 
      gives a partial order among pairs. 
    \item \label{item:overlap}
      We say that $(a,F)$ and $(b,F)$ \emph{overlap} if
      $a-b \in \ZZ F$, equivalently, if $(a+ \NN F) \cap (b + \NN F)
      \neq \varnothing$. Overlapping is an equivalence
      relation among pairs. We emphasize that overlapping is only defined
      for pairs that belong to the same face. The \emph{overlap class} containing $(a,F)$ is denoted $[a,F]$.
      
    \item \label{item:div}
    We say that $(a,F)$ \emph{divides} $(b,G)$ if there
    is $c \in \NN A$ such that $a+c+\NN F \subseteq b+\NN
    G$. This extends the notion of divisibility in $\field[\NN A]$
    (see Notation~\ref{notation:divisibility}) to the pairs of $A$.
    \end{enumerate}
\end{definition}

Overlapping is a special case of divisibility, which means that
divisibility is not an antisymmetric relation, and therefore not a
partial order on pairs. This difficulty is resolved if we
extend the definition of divisibility to overlap classes of pairs.

\begin{lemma}
\label{lemma:divPartialOrder}
Suppose $(a,F)$ divides $(b,G)$. If $(a',F)$ overlaps
$(a,F)$ and $(b',G)$ overlaps $(b,G)$, then $(a',F)$ divides $(b',G)$. In particular, divisibility is a partial order on overlap classes.
\end{lemma}

\begin{proof}
Since $(a,F)$ and $(a',F)$ overlap, we may choose $c_1 \in \NN F$ such
that $a'+c_1 \in a+\NN F$, which implies that
$a'+c_1+\NN F \subseteq a+ \NN F$. As $(a,F)$ divides $(b,G)$, there
is $c_2 \in \NN A$ such that $a+c_2+\NN F \subseteq b+ \NN G$. But
then $a'+c_1+c_2 + \NN F \subset b+\NN G$. Finally, select $c_3 \in
\NN G$ such that $b + c_3 \in b' +\NN G$. Then $a'+c_1+c_2+c_3 +\NN F \subseteq
b' + \NN G$. It follows that divisibility is well defined on overlap
classes of pairs of $A$.
Showing that divisibility is a partial order, in this case, is similarly straightforward.
\end{proof}

\subsection{Standard pairs}

We are now ready to relate the combinatorial notion of pairs to the
algebraic context of monomial ideals.

Let $I$ be a monomial ideal in $\field[\NN A]$. The \emph{standard
  monomials} of $I$ are the monomials in $\field[\NN A]$ that do not
belong to $I$. We denote
\begin{equation}
  \label{eqn:stdM}
  \stdMono(I) = \{ a \in \NN A \mid t^a \notin I \} .
\end{equation}

\begin{definition}
\label{def:standardPairs}
Let $I$ be a monomial ideal in $\field[\NN A]$. A \emph{proper
  pair} of $I$ is a pair $(a,F)$ of $A$ such that $a + \NN F \subseteq \stdMono(I)$.
A \emph{standard pair} of $I$ is a proper pair which is maximal
with respect to $\prec$
(Definition~\ref{def:Pairs}.\ref{item:prec}). The collection of
standard pairs of a monomial ideal $I$ is denoted $\stdPairs(I)$.
\end{definition}

This is the natural extension of the original definition of standard
pairs from~\cite{MR1339920}, although the partial order is reversed.  On the other
hand, standard pairs exhibit behaviors over semigroup rings that do
not occur over polynomial rings, as can be seen in the following examples.

\begin{example}[Continuation of Example~\ref{ex:ideals}]
  \label{ex:ideals2}
  \
\begin{enumerate}[leftmargin=*]
\item 
Consider $\langle x^3y, xy^2 \rangle$ in
$\field[x,y]$. Denote $F=\{(1,0)\}$, $G=\{(0,1)\},$ and $O =
\varnothing$. These subsets of the (columns of) A respectively span
the nonnegative $x$-axis, the nonnegative $y$-axis, and the origin, which are the proper faces of $\RR_{\geq 0} A$.
Our ideal has four standard pairs,
$((0,0),F)$, $((0,0),G)$, $((1,1),O)$, and $((2,1),O)$, depicted in
Figure \ref{fig:2dpolyring_std} using thick lines. In this case,
$((1,1),O)$ divides $ ((2,1),O)$. Thus there are three
maximal standard pairs with respect to divisibility. There are no
overlapping standard pairs.
Consider again the ideal $\langle x^2y^2, x^3y \rangle$ in
  $\field[x,xy,xy^2]$ from Example \ref{ex:ideals} (\ref{item:2dnormal}).
This ideal also has four standard pairs:
$((0,0),G)$, $((1,1),G)$, $((0,0),F)$, and
$((2,1),O)$ depicted in Figure
\ref{fig:2dnormal_std}. Here $F=\{(1,0)\}$, $G = \{(1,2)\}$, and $O =
\varnothing$ correspond to the proper faces of the cone $\RR_{\geq 0}A$.
The standard pair $((0,0),G)$ divides
$((1,1),G)$, and we again have three maximal standard pairs
with respect to divisibility. There are no overlapping standard pairs.
\item
  \label{ex:3dstd}
This example illustrates that overlapping standard pairs can occur
even if the semigroup ring is normal.
Consider  $\langle x^2 z^2, x^2yz^2, x^2yz^2 \rangle$ in $\field[z,xz,yz,xyz]$.
Let $F = \{(0,0,1),(0,1,1)\}$, which gives the face of $\RR_{\geq 0}A$ whose linear span is the $yz$-plane.
In this case, we have three standard pairs $((0,0,0),F)$ (a blue region
in Figure \ref{fig:3dnormal_std}), $((1,0,1),F)$ (a yellow region in
Figure \ref{fig:3dnormal_std}), and $((1,1,1),F)$ (a red region in
Figure \ref{fig:3dnormal_std}). The standard pairs $((1,0,1),F)$ and
$((1,1,1),F)$ overlap. In this case, there are two overlap classes of
standard pairs. As the pair $((0,0,0),F)$ divides (the overlap class
of) $((1,0,1),F)$, we have only one overlap class which is
maximal with respect to divisibility.
\item
Recall the ideal $\langle x,xyz,xyz^{2} \rangle$ in
$\field[x,xy,xz,xyz,y^2,z^2]$ from Example \ref{ex:ideals} (\ref{item:3dnonnormal}). 
Note again that this semigroup ring is not normal.
Let $F=\{(0,0,2),(0,2,0)\}$ be the face of $\RR_{\geq 0}
A$ whose linear span is the $yz$-plane, and let $G =\{ (0,2,0)\}$ be the
face whose linear span is the $y$-axis. In this case, our monomial
ideal has three standard pairs: $((0,0,0),F)$ (yellow points in Figure
\ref{fig:3dnonnormal_std}), $((1,0,1),F)$ (red points in Figure
\ref{fig:3dnonnormal_std}), and $((1,1,0),G)$ (blue points in Figure
\ref{fig:3dnonnormal_std}). Since $((1,1,0),G)$ cannot divide
$((1,0,1),F)$, it follows that there are two standard pairs that are
maximal with respect to divisibility. A feature of this example is
that the Zariski closure of the set $(1,0,1)+\NN F$ contains
$(1,1,0)+\NN G$, a situation that does not occur for standard pairs of
monomial ideals in polynomial rings.
\end{enumerate}
\end{example}

\begin{figure*}[t!]
    \centering
    \begin{subfigure}[t]{\twofigure}
        \centering
        \begin{tikzpicture}[scale=0.7]
	\fill[idealregioncolor] (1,6) -- (1,2) -- (3,2) -- (3,1) -- (6,1) -- (6,6) -- cycle ;
	  \draw[step=1cm,gray,very thin] (-1,-1) grid (6,6);
	  \draw [<->] (0,7) node (yaxis) [above] {$y$}
       	 |- (7,0) node (xaxis) [right] {$x$};
	\draw (-1.5,0) -- (0,0);
	\draw (0,-1.5) -- (0,0);

	\foreach \x in {0,...,6}
	\foreach \y in {0,...,6}
	\draw[black,fill=stdmcolor] (\x,\y) circle (3 pt);

	\foreach \x in {3,...,6}
	\foreach \y in {1,...,6}
	\draw[black,fill=idealcolor] (\x,\y) circle (3 pt);

	\foreach \x in {1,...,6}
	\foreach \y in {2,...,6}
	\draw[black,fill=idealcolor] (\x,\y) circle (3 pt);
	
	\draw[red,ultra thick,->] (0,0) -- (0,7);
	\draw[violet,ultra thick,->] (0,0) -- (7,0);
	\draw[green!70!black,ultra thick] (2,1) circle (10 pt);
	\draw[green!70!black,ultra thick] (1,1) circle (10 pt);
	\end{tikzpicture}
	 \caption{Standard pairs of $I= \langle x^3y, xy^2 \rangle$ in $\field[x,y]$}
	 \label{fig:2dpolyring_std}
    \end{subfigure}
    ~ 
    \begin{subfigure}[t]{\twofigure}
        \centering
        \begin{tikzpicture}[scale=0.7]
	\fill[idealregioncolor] (12,2) -- (16,2) -- (16,6) -- (14,6) -- cycle ;
	\fill[idealregioncolor] (13,1) -- (16,1) -- (16,6) -- (15.5,6) -- cycle ;

	  \draw[step=1cm,gray,very thin] (9,-1) grid (16,6);
	  \draw [<->] (10,7) node (yaxis) [above] {$t$}
       	 |- (17,0) node (xaxis) [right] {$s$};
	\draw (8.5,0) -- (10,0);
	\draw (10,-1.5) -- (10,0);
	\draw[black,<->](13.5,7) -- (10,0) -- (16,0)  ;

	\foreach \x in {10,...,16}
	\draw[black,fill=stdmcolor] (\x,0) circle (3 pt);
	\foreach \x in {10,...,13}
	{
		\draw[black,fill=stdmcolor] (\x,2*\x-20) circle (3 pt);
	};
	\foreach \x in {10,...,12}
	{
		\draw[black,fill=stdmcolor] (\x+1,1) circle (3 pt);
		\draw[black,fill=stdmcolor] (\x+1,2*\x-20+1) circle (3 pt);
	};
	\foreach \x in {10,...,14}
	\draw[black,fill=idealcolor] (2+\x,2) circle (3 pt);
	\foreach \x in {10,...,13}
	\draw[black,fill=idealcolor] (3+\x,1) circle (3 pt);
	\foreach \x in {13,...,16}
	\foreach \y in {3,4}
	\draw[black,fill=idealcolor] (\x,\y) circle (3 pt);
	\foreach \x in {14,...,16}
	\foreach \y in {5,6}
	\draw[black,fill=idealcolor] (\x,\y) circle (3 pt);
	
	\draw[red,ultra thick,->] (10,0) -- (13.5,7);
	\draw[red,ultra thick,->] (11,1) -- (14,7);
	\draw[violet,ultra thick,->] (10,0) -- (17,0);
	\draw[green!70!black,ultra thick] (12,1) circle (10 pt);
	\end{tikzpicture}
        \caption{Standard pairs of $\langle x^2y^2, x^3y \rangle$ in $\field[x,xy,xy^2]$}
        \label{fig:2dnormal_std}
    \end{subfigure}
    \caption{Standard pairs in two-dimensional affine semigroup rings}
    \label{fig:2dafffinesemigp_std}
\end{figure*}
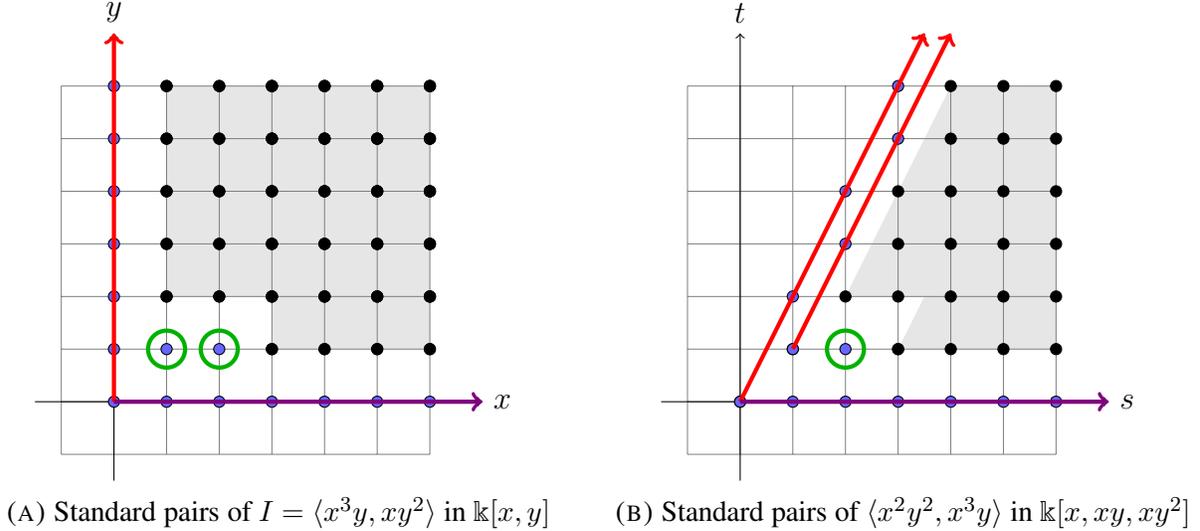

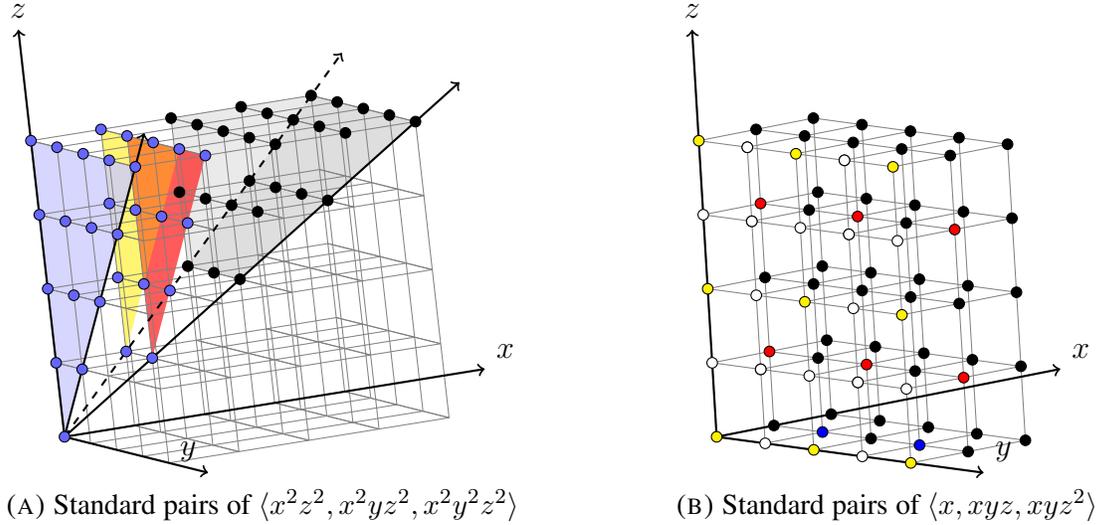
\begin{figure*}[t!]
    \centering
    \begin{subfigure}[t]{\twofigure}
        \centering
	 \tdplotsetmaincoords{90}{90} 
        \begin{tikzpicture}[tdplot_main_coords, scale=1]
	\tdplotsetrotatedcoords{40}{-10}{30}
	\fill[black!20, fill opacity=0.4,tdplot_rotated_coords] (2,0,2)--  (2,2,2) --  (4,4,4) -- (4,0,4) -- cycle;
	\fill[black!20, fill opacity=0.4,tdplot_rotated_coords] (2,4,4) -- (4,4,4) -- (4,0,4) -- (2,0,4)--  cycle;
	\fill[black!20, fill opacity=0.4,tdplot_rotated_coords] (2,4,4) -- (2,0,4)--  (2,0,2) -- (2,2,2) -- cycle;
	\fill[black!20, fill opacity=0.4,tdplot_rotated_coords] (4,4,4) -- (2,4,4)--  (2,2,2) -- cycle;

	\fill[yellow!80, fill opacity=0.8,tdplot_rotated_coords] (1,0,1) -- (1,3,4) -- (1,0,4) -- cycle;
	\fill[red!80, fill opacity=0.8,tdplot_rotated_coords] (1,1,1) -- (1,4,4) -- (1,1,4) -- cycle;
	\fill[orange!80, fill opacity=0.8,tdplot_rotated_coords] (1,1,2) -- (1,1,4) -- (1,3,4) -- cycle;
	\fill[blue!20,fill opacity=0.8,tdplot_rotated_coords] (0,0,0) -- (0,4,4) -- (0,0,4) -- cycle;

	\foreach \x in {0,...,4}
	{
		\draw[step=1cm,gray,very thin, tdplot_rotated_coords] (\x,0,0) -- (\x,4,0);
		\draw[step=1cm,gray,very thin, tdplot_rotated_coords] (0,\x,0) -- (4,\x,0);
		\draw[step=1cm,gray,very thin, tdplot_rotated_coords] (\x,0,1) -- (\x,4,1);
		\draw[step=1cm,gray,very thin, tdplot_rotated_coords] (0,\x,1) -- (4,\x,1);
		\draw[step=1cm,gray,very thin, tdplot_rotated_coords] (\x,0,2) -- (\x,4,2);
		\draw[step=1cm,gray,very thin, tdplot_rotated_coords] (0,\x,2) -- (4,\x,2);
		\draw[step=1cm,gray,very thin, tdplot_rotated_coords] (\x,0,3) -- (\x,4,3);
		\draw[step=1cm,gray,very thin, tdplot_rotated_coords] (0,\x,3) -- (4,\x,3);
		\draw[step=1cm,gray,very thin, tdplot_rotated_coords] (\x,0,4) -- (\x,4,4);
		\draw[step=1cm,gray,very thin, tdplot_rotated_coords] (0,\x,4) -- (4,\x,4);

		\draw[step=1cm,gray,very thin, tdplot_rotated_coords] (0,\x,0) -- (0,\x,4);
		\draw[step=1cm,gray,very thin, tdplot_rotated_coords] (0,0,\x) -- (0,4,\x);
		\draw[step=1cm,gray,very thin, tdplot_rotated_coords] (1,\x,0) -- (1,\x,4);
		\draw[step=1cm,gray,very thin, tdplot_rotated_coords] (1,0,\x) -- (1,4,\x);
		\draw[step=1cm,gray,very thin, tdplot_rotated_coords] (2,\x,0) -- (2,\x,4);
		\draw[step=1cm,gray,very thin, tdplot_rotated_coords] (2,0,\x) -- (2,4,\x);
		\draw[step=1cm,gray,very thin, tdplot_rotated_coords] (3,\x,0) -- (3,\x,4);
		\draw[step=1cm,gray,very thin, tdplot_rotated_coords] (3,0,\x) -- (3,4,\x);
		\draw[step=1cm,gray,very thin, tdplot_rotated_coords] (4,\x,0) -- (4,\x,4);
		\draw[step=1cm,gray,very thin, tdplot_rotated_coords] (4,0,\x) -- (4,4,\x);
	};
	
	\draw[thick,->,tdplot_rotated_coords] (0,0,0) -- (6,0,0) node[anchor=south west]{$x$}; 
	\draw[thick,->,tdplot_rotated_coords] (0,0,0) -- (0,5.5,0) node[anchor=south east]{$y$}; 
	\draw[thick,->,tdplot_rotated_coords] (0,0,0) -- (0,0,5.5) node[anchor=south]{$z$};
	\draw[thick,->,tdplot_rotated_coords] (0,0,0) -- (0,4.5,4.5);
	\draw[thick,dashed,->,tdplot_rotated_coords] (0,0,0) -- (4.5,0,4.5);
	\draw[thick, ->,tdplot_rotated_coords] (0,0,0) -- (4.5,4.5,4.5);

	\foreach \y in {0,...,4}
	\draw[black,fill=stdmcolor,tdplot_rotated_coords] (0,0,\y) circle (2 pt);
	\draw[black,fill=stdmcolor,tdplot_rotated_coords] (0,1,1) circle (2 pt);
	\draw[black,fill=stdmcolor,tdplot_rotated_coords] (0,1,2) circle (2 pt);
	\draw[black,fill=stdmcolor,tdplot_rotated_coords] (0,2,2) circle (2 pt);
	\draw[black,fill=stdmcolor,tdplot_rotated_coords] (0,1,3) circle (2 pt);
	\draw[black,fill=stdmcolor,tdplot_rotated_coords] (0,2,3) circle (2 pt);
	\draw[black,fill=stdmcolor,tdplot_rotated_coords] (0,3,3) circle (2 pt);
	\draw[black,fill=stdmcolor,tdplot_rotated_coords] (0,1,4) circle (2 pt);
	\draw[black,fill=stdmcolor,tdplot_rotated_coords] (0,2,4) circle (2 pt);
	\draw[black,fill=stdmcolor,tdplot_rotated_coords] (0,3,4) circle (2 pt);
	\draw[black,fill=stdmcolor,tdplot_rotated_coords] (0,4,4) circle (2 pt);

	\foreach \y in {0,...,3}
	\draw[black,fill=stdmcolor,tdplot_rotated_coords] (1,0,1+\y) circle (2 pt);

	\foreach \y in {0,...,3}
	\draw[black,fill=stdmcolor,tdplot_rotated_coords] (1,1+\y,1+\y) circle (2 pt);

	\draw[black,fill=stdmcolor,tdplot_rotated_coords] (1,1,2) circle (2 pt);
	\draw[black,fill=stdmcolor,tdplot_rotated_coords] (1,1,3) circle (2 pt);
	\draw[black,fill=stdmcolor,tdplot_rotated_coords] (1,2,3) circle (2 pt);
	\draw[black,fill=stdmcolor,tdplot_rotated_coords] (1,1,4) circle (2 pt);
	\draw[black,fill=stdmcolor,tdplot_rotated_coords] (1,2,4) circle (2 pt);
	\draw[black,fill=stdmcolor,tdplot_rotated_coords] (1,3,4) circle (2 pt);

	\draw[black,fill=idealcolor,tdplot_rotated_coords] (2,0,2) circle (2 pt);
	\draw[black,fill=idealcolor,tdplot_rotated_coords] (2,1,2) circle (2 pt);
	\draw[black,fill=idealcolor,tdplot_rotated_coords] (2,2,2) circle (2 pt);
	\draw[black,fill=idealcolor,tdplot_rotated_coords] (2,0,3) circle (2 pt);
	\draw[black,fill=idealcolor,tdplot_rotated_coords] (2,1,3) circle (2 pt);
	\draw[black,fill=idealcolor,tdplot_rotated_coords] (2,2,3) circle (2 pt);
	\draw[black,fill=idealcolor,tdplot_rotated_coords] (2,3,3) circle (2 pt);
	\draw[black,fill=idealcolor,tdplot_rotated_coords] (3,0,3) circle (2 pt);
	\draw[black,fill=idealcolor,tdplot_rotated_coords] (3,1,3) circle (2 pt);
	\draw[black,fill=idealcolor,tdplot_rotated_coords] (3,2,3) circle (2 pt);
	\draw[black,fill=idealcolor,tdplot_rotated_coords] (3,3,3) circle (2 pt);
	\draw[black,fill=idealcolor,tdplot_rotated_coords] (2,0,4) circle (2 pt);
	\draw[black,fill=idealcolor,tdplot_rotated_coords] (2,1,4) circle (2 pt);
	\draw[black,fill=idealcolor,tdplot_rotated_coords] (2,2,4) circle (2 pt);
	\draw[black,fill=idealcolor,tdplot_rotated_coords] (2,3,4) circle (2 pt);
	\draw[black,fill=idealcolor,tdplot_rotated_coords] (2,4,4) circle (2 pt);
	\draw[black,fill=idealcolor,tdplot_rotated_coords] (3,0,4) circle (2 pt);
	\draw[black,fill=idealcolor,tdplot_rotated_coords] (3,1,4) circle (2 pt);
	\draw[black,fill=idealcolor,tdplot_rotated_coords] (3,2,4) circle (2 pt);
	\draw[black,fill=idealcolor,tdplot_rotated_coords] (3,3,4) circle (2 pt);
	\draw[black,fill=idealcolor,tdplot_rotated_coords] (3,4,4) circle (2 pt);
	\draw[black,fill=idealcolor,tdplot_rotated_coords] (4,0,4) circle (2 pt);
	\draw[black,fill=idealcolor,tdplot_rotated_coords] (4,1,4) circle (2 pt);
	\draw[black,fill=idealcolor,tdplot_rotated_coords] (4,2,4) circle (2 pt);
	\draw[black,fill=idealcolor,tdplot_rotated_coords] (4,3,4) circle (2 pt);
	\draw[black,fill=idealcolor,tdplot_rotated_coords] (4,4,4) circle (2 pt);

	\end{tikzpicture}
	 \caption{Standard pairs of $\langle x^2 z^2, x^2yz^2, x^2y^2z^2 \rangle$}
	 \label{fig:3dnormal_std}
    \end{subfigure}
    ~ 
    \begin{subfigure}[t]{\twofigure}
        \centering
       \tdplotsetmaincoords{90}{90} 
	\begin{tikzpicture}[tdplot_main_coords, scale=1]
	\tdplotsetrotatedcoords{20}{-10}{30}
	\foreach \x in {0,...,4}
	{
	\draw[step=1cm,gray,very thin, tdplot_rotated_coords] (0,\x,0) -- (0,\x,4);
	\draw[step=1cm,gray,very thin, tdplot_rotated_coords] (0,0,\x) -- (0,4,\x);
	\draw[step=1cm,gray,very thin, tdplot_rotated_coords] (1,\x,0) -- (1,\x,4);
	\draw[step=1cm,gray,very thin, tdplot_rotated_coords] (1,0,\x) -- (1,4,\x);
	\draw[step=1cm,gray,very thin, tdplot_rotated_coords] (2,\x,0) -- (2,\x,4);
	\draw[step=1cm,gray,very thin, tdplot_rotated_coords] (2,0,\x) -- (2,4,\x);
	};
	\foreach \y in {0,...,4}
	\foreach \z in {0,...,4}
	\draw[step=1cm,gray,very thin, tdplot_rotated_coords] (0,\y,\z) -- (2,\y,\z);

	\draw[thick,->,tdplot_rotated_coords] (0,0,0) -- (6,0,0) node[anchor=south west]{$x$}; 
	\draw[thick,->,tdplot_rotated_coords] (0,0,0) -- (0,5.5,0) node[anchor=south west]{$y$}; 
	\draw[thick,->,tdplot_rotated_coords] (0,0,0) -- (0,0,5.5) node[anchor=south]{$z$};
	\foreach \y in {0,...,2}
		\foreach \z in {0,...,2}
		{
			\draw[black,fill=black,tdplot_rotated_coords] (1,2*\y,2*\z) circle (2 pt);
		};
	\foreach \y in {0,...,1}
		\foreach \z in {0,...,1}
			\draw[black,fill=black,tdplot_rotated_coords] (1,1+2*\y,1+2*\z) circle (2 pt);

	\foreach \y in {0,...,1}
		\foreach \z in {0,...,1}
			\draw[black,fill=black,tdplot_rotated_coords] (1,1+2*\y,2+2*\z) circle (2 pt);
	\foreach \y in {0,...,4}
		\foreach \z in {0,...,4}
			\draw[black,fill=black,tdplot_rotated_coords] (2,\y,\z) circle (2 pt);

	\foreach \y in {0,...,1}
	\foreach \z in {0,...,1}
	{
		\draw[black,fill=white,tdplot_rotated_coords] (0,2*\y+1,2*\z) circle (2 pt);
		\draw[black,fill=white,tdplot_rotated_coords] (0,2*\y+1,1) circle (2 pt);
		\draw[black,fill=white,tdplot_rotated_coords] (0,2*\y,2*\z+1) circle (2 pt);
		\draw[black,fill=white,tdplot_rotated_coords] (0,2*\y+1,3) circle (2 pt);
		\draw[black,fill=white,tdplot_rotated_coords] (0,2*\y+1,4) circle (2 pt);
	};
	\draw[black,fill=white,tdplot_rotated_coords] (0,4,1) circle (2 pt);
	\draw[black,fill=white,tdplot_rotated_coords] (0,4,3) circle (2 pt);

	\foreach \z in {0,...,2}
	\foreach \y in {0,...,2}
	{
		\draw[black,fill=yellow,tdplot_rotated_coords] (0,2*\y,2*\z) circle (2 pt);
	};
	\foreach \z in {0,...,1}
	\foreach \y in {0,...,2}
	{
		\draw[black,fill=red,tdplot_rotated_coords] (1,2*\y,1+2*\z) circle (2 pt);
	};
	\foreach \y in {0,...,1}
	{
		\draw[black,fill=blue,tdplot_rotated_coords] (1,1+2*\y,0) circle (2 pt);
	};
	\end{tikzpicture}
        \caption{Standard pairs of $\langle x,xyz,xyz^{2} \rangle$}
        \label{fig:3dnonnormal_std}
    \end{subfigure}
    \caption{Standard pairs in three-dimensional affine semigroup rings}
    \label{fig:3dafffinesemigp_std}
\end{figure*}

\subsection{Primary Decomposition}

Our goal now is to use standard pairs to give a primary
decomposition of a monomial ideal $I$ in $\field[\NN A]$ with monomial
primary components. This is achieved in
Theorem~\ref{thm:primaryDecomposition}, whose proof we break into
several steps.

First, we give a sufficient condition for a monomial ideal to be primary.

\begin{proposition}
\label{prop:primary}
Let $I$ be a monomial ideal in $\field[\NN A]$. If all the standard
pairs of $I$ belong to the same face $F$ of $A$, then $I$ is $\pf$-primary.
\end{proposition}

\begin{proof}
This proof has two parts. We first show that $\pf$ is an associated
prime of $I$, and then show that no other prime is associated.

Since $I$ is $\ZZ^d$-homogeneous, so are all of its associated
primes, which means that the only possible associated primes are of
the form $\pg$ for some face $G$ of $\RR_{\geq 0} A$. Moreover, a
prime $\pg$ is associated to $I$ if and only if $(I:t^a) = \pg$ for
some monomial $t^a$, $a \in \NN A$.

The assumption on the standard pairs means that, for any $b
\in \NN A$ such that $t^b \notin I$, we have $b+\NN F \subset
\stdMono(I)$. In ideal-theoretic terms, this means that $(I:t^b)
\subseteq \pf$.

Let $(a,F)$ be standard pair of $I$ whose overlap class $[a,F]$ is maximal
with respect to divisibility. Recall that we denote $[a,F]$ the overlap class containing $(a,F)$. We claim that if $b \in a+\NN F$, then
$(I:t^b)=\pf$. To see this, let $c \in \NN A \minus \NN F$. If
$t^{b+c} \notin I$, then $b+c$ belongs to $a'+\NN F$ for some standard
pair $(a',F)$ of $I$ (all standard pairs of $I$ belong to $F$). As $c \notin \NN F$, this contradicts the
maximality of $[a,F]$. We conclude that if 
$c \in \NN A \minus \NN F$, then $t^{b+c}\in I$, so that $(I:t^b)
\supset \pf$. We already knew the reverse inclusion, therefore 
$(I:t^b) = \pf$, which shows that $\pf$ is associated to $I$.

To see that no other prime is associated we show that $(I:t^{b})$ is not prime ideal in all other cases. If the overlap class of
$(a,F)$ is not maximal with respect to divisibility and $b \in a + \NN
F$, then $(I:t^b)$ is not prime. Since $[a,F]$ is not maximal, there
is a standard pair $(a',F)$ of $I$, whose overlap class is maximal
with respect to divisibility, and such that $(a,F)$ divides $(a',F)$.
In particular, there is $c \in \NN A$ such that $b+c \in a'+\NN
F$. Indeed, $c \notin \NN F$ as $(a,F)$ and $(a'F)$ are not in the
same overlap class. Since $(a',F)$ is a standard pair of $I$, it follows that $t^c \notin (I:t^b)$. By the previous argument, however, since $c \in \NN A \minus \NN F$ and $b+c \in a'+\NN F$, we have $t^c \in (I:t^{b+c})$, which implies
$t^{2c} \in (I:t^b)$. We conclude that $(I:t^b)$ is not prime.
\end{proof}

The converse of Proposition~\ref{prop:primary} holds and is proved by exhibiting a
primary decomposition (Theorem~\ref{thm:primaryDecomposition}). Before
we can do that, we prove a finiteness result.

\begin{lemma}
  \label{lemma:finitelyManyMaximalClasses}
Let $I$ be a monomial ideal in $\field[\NN A]$, and let $F$ be a face
of $A$ such that $I$ has a standard pair belonging to $F$. There are
finitely many overlap classes of standard pairs of $I$ belonging to
$F$ that are maximal with respect to divisibility.
\end{lemma}

\begin{proof}
  If $(a,F)$ and $(b,F)$ are standard pairs of $I$ that do not
  overlap, and whose overlap classes are maximal with respect to
  divisibility, then $a-b \notin \NN A$ and $b-a \notin \NN A$. Now
  apply Lemma~\ref{lemma:Dickson}.
\end{proof}

Our next step is to construct a $\pf$-primary ideal, which is later
shown to be a valid choice for a $\pf$-primary component of $I$.

\begin{proposition}
\label{prop:component}
Let $I$ be a monomial ideal in $\field[\NN A]$, and let $F$ be a face
of $A$ such that $I$ has a standard pair belonging to $F$.
Set  
\[
  S = \left\{ t^b \; \bigg| 
      \begin{array}{l}
        b \in \NN A \text{ divides some element of } a+ \NN F
        \text{ for some standard pair } \\
        (a,F) \text{ of } I \text{ whose
        overlap class is maximal with respect to divisibility}
      \end{array}
\right\} .
\]
Then $S$ is the set of standard 
monomials of a monomial ideal $C_F$, $C_F \supset I$, and $C_F$ is
$\pf$-primary. 
\end{proposition}

\begin{proof}
The first assertion is equivalent to the following 
statement, whose proof is straightforward: if $t^b$ does not divide any monomial arising
from the maximal overlap classes, and $c \in \NN A$, then $t^{b+c}$ cannot
divide any such monomial either. The second assertion follows from the
fact that $I$ is an ideal: no monomial in $S$ can belong to $I$.

It remains to be shown that $C_F$ is $\pf$-primary.
By Proposition~\ref{prop:primary}, it is enough to
show that all standard pairs of $C_F$ belong to the face $F$. To
see this, we first observe that if $t^b \in S$, then $t^{b+c} \in S$
for all $c \in \NN F$. This implies that $(b,F)$ is a proper pair
for all $t^b \in S$.

To finish the proof, we check that $C_F$ has no proper pairs of the form
$(b,G)$, where $G$ strictly contains $F$. This is a consequence of the
following claim:
\begin{quote}
If $t^b \in S$ and $c \in \NN A \minus \NN F$, 
there is a positive integer $k$ such that $t^{b+k c} \notin
S$.
\end{quote}

To prove the claim, as $c \in \NN A \minus \NN F$, there is a facet $H$ of $A$ such that
$H$ contains $F$ and $\varphi_H(c)>0$ (see
Definition~\ref{def:integerDistance}). 

Since $H$ contains $F$, $\varphi_H$ is constant on each set $a +\NN F$.
Moreover, if $(a,F)$ and $(a',F)$ are overlapping standard pairs of
$I$, then the value of $\varphi_H$ on $a+\NN F$
equals the value of $\varphi_H$ on $a'+\NN F$. Now, by
Lemma~\ref{lemma:finitelyManyMaximalClasses} there are
finitely many maximal overlap classes of standard pairs. This implies
that there is a positive integer $N$ which is an upper bound for the values
that $\varphi_H$ attains on these classes.
It follows that for any monomial in $S$, the
value of $\varphi_H$ on its exponent is at most $N$.
In particular, $\varphi_H(b) \leq N$. Since $\varphi_H(c)>0$, we may
choose a sufficiently large $k$ so that $\varphi_H(b+kc) =
\varphi_H(b)+k\varphi_H(c)>N$. It follows that $b+kc \notin S$, as was claimed.
\end{proof}

\begin{theorem}
  \label{thm:primaryDecomposition}
  Let $I$ be a monomial ideal in $\field[\NN A]$. Let
\[
  \mathscr{S} =\{ F \text{ face of } A \mid I \text{ has a standard
    pair belonging to } F\}.
\]
  For $F \in \mathscr{S}$, let $C_F$ be
  as in Proposition~\ref{prop:component}. Then $I = \cap_{F \in
    \mathscr{S}} C_F$ is an irredundant primary decomposition of $I$.
  Consequently,
  \begin{enumerate}
    \item  $\pf$ is associated to $I$ if and only if $I$ has a
      standard pair that belongs to $F$.
    \item $I$ is $\pf$-primary if and only if all standard pairs of
      $I$ belong to $F$.
   \end{enumerate}
\end{theorem}

\begin{proof}
By Proposition~\ref{prop:component} it is enough to show that $I = \cap_{F \in
    \mathscr{S}} C_F$. By construction, $I \subseteq C_F$ for all
  $F\in \mathscr{S}$, so that $I \subseteq \cap_{F \in
    \mathscr{S}} C_F$. To see the reverse inclusion, we claim that if
  $t^b \notin I$, then $t^b \notin \cap_{F \in \mathscr{S}} C_F$, or
  equivalently, $t^b \notin C_F$ for some $F\in \mathscr{S}$.  To see this, since $t^b \notin I$, there is a standard pair $(a,F)$ of $I$ such that $b \in a +\NN F$. But then $t^b \notin C_F$ by the construction of $C_F$.
  
To show the irredundancy, suppose that the decomposition is not irredundant. Then, there exists a face $F$ such that $C_{F} \supseteq \bigcap_{G \neq F}C_{G} \supseteq I$. Let $[a,F]$ be a maximal overlap class of $I$. Denote $\bigcup[a,F]$ be the union of all monomials in $a'+\NN F$ where $(a',F)$ is in the overlap class $[a,F]$. Then, for any monomial $b \in \bigcup[a,F]$, there is a face $G$ such that $b \not\in C_{G}.$ Hence, $b$ divides some element of $c+\NN G$ such that $(c,G)$ form a standard pair of an overlap class which is maximal with respect to divisibility. This implies two facts; first of all, $F$ is not a vertex. If $F$ is a vertex, then $(a',F)$ is inside of the standard pair, contradicting the maximality of the standard pair. Next, $(b,G)$ is a proper pair of $I$, hence it is contained in $\bigcup[b,G]$. From Lemma \ref{lemma:finitelyManyMaximalClasses}, $(\bigcup[a,F])$ is covered by by finitely many $(\bigcup[d,G'])$ for some faces $G'$ with monomials $d$.
  
Now we show that every overlap class covers finitely many monomials of $(\bigcup[a,F])$. This contradicts pigeonhole principle since $\bigcup[a,F]$ is infinite. Suppose that $[d,G']$ is such that $(\bigcup[a,F]) \cap (\bigcup[d,G'])$ is infinite. This implies $a+c \in d + \NN G'$ for some $c \in \NN F$. If $F \subseteq G'$, then $(a,G')$ is a proper pair containing $(a,F)$, a contradiction. Conversely, if $G' \subseteq F$, then there exists $e \in \NN F$ such that $d+e \in I$. Thus, $a+c + e = d+c' +e \in I$, a contradiction. Thus, there exist two hyperplanes $H_{F}$ and $H_{G}$ such that $\varphi_{H_{F}}(\NN F)$ is bounded while $\varphi_{H_{F}}(\NN G)$ is not, and vice versa. Since $(\bigcup[a,F]) \cap (\bigcup[d,G'])$ is infinite, let $\{ f_{n}\} \subseteq \NN F$ be an ordered sequence such that $\varphi_{H_{G}}(f_{i}) < \varphi_{H_{G}}(f_{i+1})$ for all $i$ and $a+f_{i} \in (\bigcup[a,F]) \cap (\bigcup[d,G'])$. Such a sequence exists, since $\{ \varphi_{H_{G}}(f) < n: f \in \NN F\}$ is always finite for any $n\in \NN$. Then, $\varphi_{H_{G}}(\{ a+f_{i}:i \in \NN\})$ is not bounded, contradicting the fact that $\varphi_{H_{G}}(\bigcup[d,G'])$ is bounded.
\end{proof}

\begin{example}[Continuation of Examples~\ref{ex:ideals}
  and~\ref{ex:ideals2}]
  \
  \label{ex:ideals3}
  \begin{enumerate}[leftmargin=*]
    \item In Example \ref{ex:ideals}(\ref{item:2dpoly}), $\langle x^3y, xy^2 \rangle \subset \field[x,y]$
      has three maximal overlap classes of standard pairs,
      $[(0,0),F]$, $[(0,0),G]$, and $[(2,1),O]$. In the notation of
      Proposition~\ref{prop:component} and
      Theorem~\ref{thm:primaryDecomposition}, $C_F =\langle xy,xy^2\rangle$,
      $C_G = \langle x \rangle$ and $C_O = \langle x^2,xy^2\rangle$,
      yielding the primary decomposition
      \[ \langle x^3y, xy^2 \rangle = \langle xy,xy^2\rangle \cap \langle x
        \rangle \cap \langle x^2,xy^2	\rangle. \]
      \item Figure~\ref{fig:2dafffinesemigp_std_primary} depicts the
        primary decomposition of $\langle x^2y^2,x^3y\rangle \subset
        \field[x,xy,xy^2]$. Ideals are indicated using shaded regions,
        standard pairs are illustrated using thick lines and circles.
        \item In Example \ref{ex:ideals}(\ref{item:3dnormal}), the ideal $\langle x^2 z^2, x^2yz^2,
          x^2yz^2 \rangle \subset \field[z,xz,yz,xyz]$ under consideration is
          $\pf$-primary.
          \item A primary decomposition of $\langle x,xyz,xyz^{2}
            \rangle \subset \field[x,xy,xz,xyz,y^2,z^2]$ is depicted
            in Figure~\ref{fig:3dafffinesemigp_std_primary}. Exponents
            of monomials in the ideal are colored black. Other colors
            are used to indicate monomials from the same standard pair.
  \end{enumerate} 
\end{example}

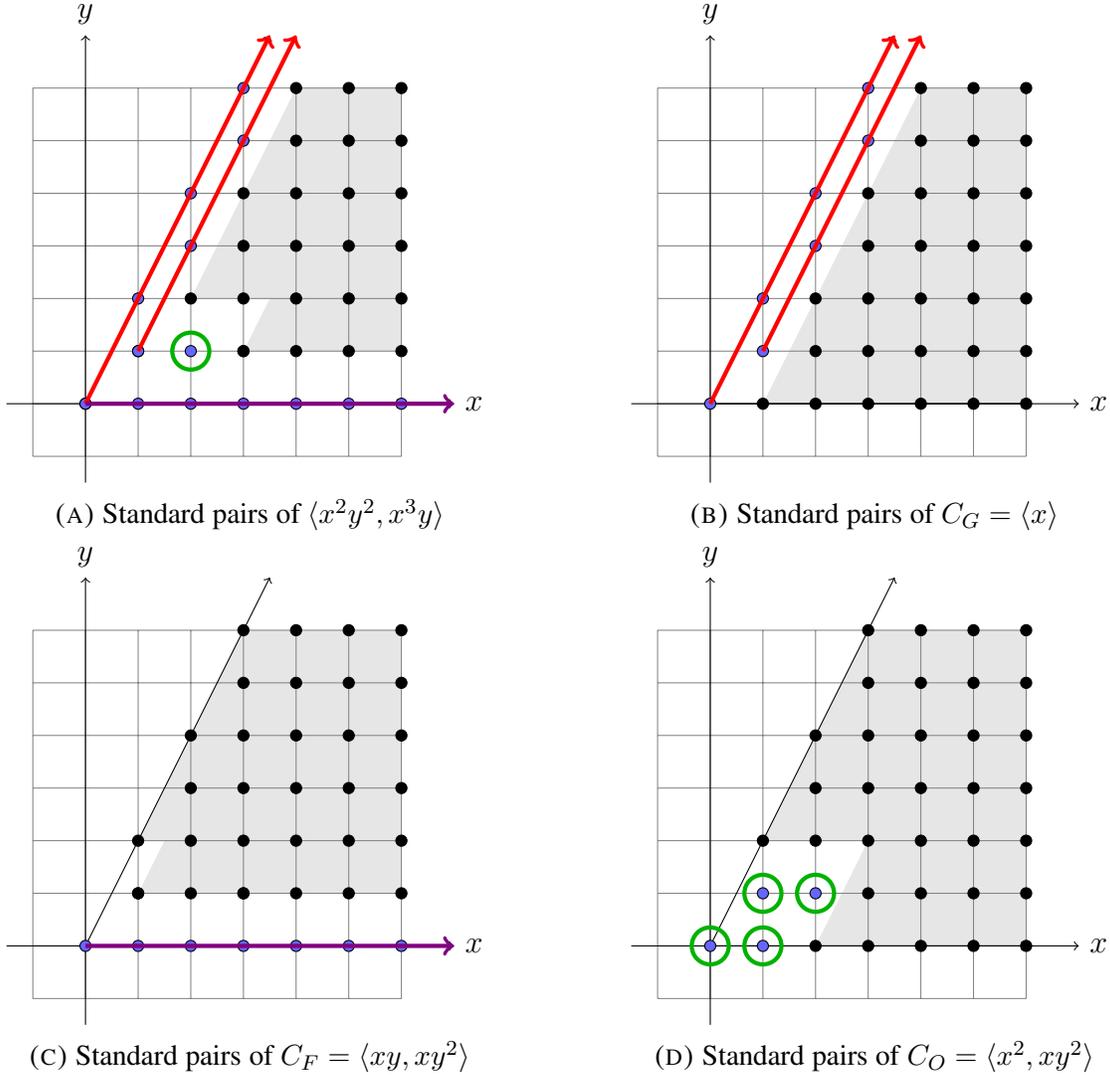
\begin{figure*}[t!]
    \centering 
    \begin{subfigure}[t]{\fourfigure}
        \centering
        \begin{tikzpicture}[scale=0.7]
	\fill[idealregioncolor] (12,2) -- (16,2) -- (16,6) -- (14,6) -- cycle ;
	\fill[idealregioncolor] (13,1) -- (16,1) -- (16,6) -- (15.5,6) -- cycle ;

	  \draw[step=1cm,gray,very thin] (9,-1) grid (16,6);
	  \draw [<->] (10,7) node (yaxis) [above] {$y$}
       	 |- (17,0) node (xaxis) [right] {$x$};
	\draw (8.5,0) -- (10,0);
	\draw (10,-1.5) -- (10,0);
	\draw[black,<->](13.5,7) -- (10,0) -- (16,0)  ;

	\foreach \x in {10,...,16}
	\draw[black,fill=stdmcolor] (\x,0) circle (3 pt);
	\foreach \x in {10,...,13}
	{
		\draw[black,fill=stdmcolor] (\x,2*\x-20) circle (3 pt);
	};
	\foreach \x in {10,...,12}
	{
		\draw[black,fill=stdmcolor] (\x+1,1) circle (3 pt);
		\draw[black,fill=stdmcolor] (\x+1,2*\x-20+1) circle (3 pt);
	};
	\foreach \x in {10,...,14}
	\draw[black,fill=idealcolor] (2+\x,2) circle (3 pt);
	\foreach \x in {10,...,13}
	\draw[black,fill=idealcolor] (3+\x,1) circle (3 pt);
	\foreach \x in {13,...,16}
	\foreach \y in {3,4}
	\draw[black,fill=idealcolor] (\x,\y) circle (3 pt);
	\foreach \x in {14,...,16}
	\foreach \y in {5,6}
	\draw[black,fill=idealcolor] (\x,\y) circle (3 pt);
	
	\draw[red,ultra thick,->] (10,0) -- (13.5,7);
	\draw[red,ultra thick,->] (11,1) -- (14,7);
	\draw[violet,ultra thick,->] (10,0) -- (17,0);
	\draw[green!70!black,ultra thick] (12,1) circle (10 pt);
	\end{tikzpicture}
        \caption{Standard pairs of $\langle x^2y^2, x^3y \rangle$}
        \label{fig:2dnormal_std_primary}
    \end{subfigure}
    ~
   \begin{subfigure}[t]{\fourfigure}
        \centering
        \begin{tikzpicture}[scale=0.7]
	\fill[idealregioncolor] (11,0) -- (16,0) -- (16,6) -- (14,6) -- cycle ;
	  \draw[step=1cm,gray,very thin] (9,-1) grid (16,6);
	  \draw [<->] (10,7) node (yaxis) [above] {$y$}
       	 |- (17,0) node (xaxis) [right] {$x$};
	\draw (8.5,0) -- (10,0);
	\draw (10,-1.5) -- (10,0);
	\draw[black,<->](13.5,7) -- (10,0) -- (16,0)  ;

	\foreach \x in {10,...,13}
	{
		\draw[black,fill=stdmcolor] (\x,2*\x-20) circle (3 pt);
	};
	\foreach \x in {10,...,12}
	{
		\draw[black,fill=stdmcolor] (\x+1,2*\x-20+1) circle (3 pt);
	};
	\foreach \x in {11,...,16}
	\draw[black,fill=idealcolor] (\x,0) circle (3 pt);
	\draw[black,fill=idealcolor] (12,1) circle (3 pt);
	\foreach \x in {10,...,14}
	\draw[black,fill=idealcolor] (2+\x,2) circle (3 pt);
	\foreach \x in {10,...,13}
	\draw[black,fill=idealcolor] (3+\x,1) circle (3 pt);
	\foreach \x in {13,...,16}
	\foreach \y in {3,4}
	\draw[black,fill=idealcolor] (\x,\y) circle (3 pt);
	\foreach \x in {14,...,16}
	\foreach \y in {5,6}
	\draw[black,fill=idealcolor] (\x,\y) circle (3 pt);
	
	\draw[red,ultra thick,->] (10,0) -- (13.5,7);
	\draw[red,ultra thick,->] (11,1) -- (14,7);
	\end{tikzpicture}
        \caption{Standard pairs of $C_{G}= \langle x \rangle$}
        \label{fig:2dnormal_std_primary12}
    \end{subfigure}
    
    \begin{subfigure}[t]{\fourfigure}
        \centering
        \begin{tikzpicture}[scale=0.7]
	\fill[idealregioncolor] (11,2) -- (16,2) -- (16,6) -- (13,6) -- cycle ;
	\fill[idealregioncolor] (11,1) -- (16,1) -- (16,6) -- (13.5,6) -- cycle ;

	  \draw[step=1cm,gray,very thin] (9,-1) grid (16,6);
	  \draw [<->] (10,7) node (yaxis) [above] {$y$}
       	 |- (17,0) node (xaxis) [right] {$x$};
	\draw (8.5,0) -- (10,0);
	\draw (10,-1.5) -- (10,0);
	\draw[black,<->](13.5,7) -- (10,0) -- (16,0)  ;

	\foreach \x in {10,...,16}
	\draw[black,fill=stdmcolor] (\x,0) circle (3 pt);
	\foreach \x in {10,...,12}
	{
		\draw[black,fill=idealcolor] (\x+1,2*\x-20+1) circle (3 pt);
	};
	\foreach \x in {11,...,13}
	{
		\draw[black,fill=idealcolor] (\x,1) circle (3 pt);
		\draw[black,fill=idealcolor] (\x,2*\x-20) circle (3 pt);
	};
	\foreach \x in {10,...,14}
	\draw[black,fill=idealcolor] (2+\x,2) circle (3 pt);
	\foreach \x in {10,...,13}
	\draw[black,fill=idealcolor] (3+\x,1) circle (3 pt);
	\foreach \x in {13,...,16}
	\foreach \y in {3,4}
	\draw[black,fill=idealcolor] (\x,\y) circle (3 pt);
	\foreach \x in {14,...,16}
	\foreach \y in {5,6}
	\draw[black,fill=idealcolor] (\x,\y) circle (3 pt);
	
	\draw[violet,ultra thick,->] (10,0) -- (17,0);
	\end{tikzpicture}
        \caption{Standard pairs of $C_{F}= \langle xy, xy^2 \rangle$}
        \label{fig:2dnormal_std_primary10}
    \end{subfigure}
    \begin{subfigure}[t]{\fourfigure}
        \centering
        \begin{tikzpicture}[scale=0.7]
	\fill[idealregioncolor] (11,2) -- (16,2) -- (16,6) -- (13,6) -- cycle ;
	\fill[idealregioncolor] (12,0) -- (16,0) -- (16,6) -- (15,6) -- cycle ;

	  \draw[step=1cm,gray,very thin] (9,-1) grid (16,6);
	  \draw [<->] (10,7) node (yaxis) [above] {$y$}
       	 |- (17,0) node (xaxis) [right] {$x$};
	\draw (8.5,0) -- (10,0);
	\draw (10,-1.5) -- (10,0);
	\draw[black,<->](13.5,7) -- (10,0) -- (16,0)  ;

	\draw[black,fill=stdmcolor] (10,0) circle (3 pt);
	\draw[black,fill=stdmcolor] (11,0) circle (3 pt);
	\draw[black,fill=stdmcolor] (11,1) circle (3 pt);
	\draw[black,fill=stdmcolor] (12,1) circle (3 pt);
	
	\foreach \x in {11,...,13}
	{
		\draw[black,fill=idealcolor] (\x,2*\x-20) circle (3 pt);
	};
	\foreach \x in {11,12}
	{
		\draw[black,fill=idealcolor] (\x+1,2*\x-20+1) circle (3 pt);
	};
	\foreach \x in {12,...,16}
	\draw[black,fill=idealcolor] (\x,0) circle (3 pt);
	\foreach \x in {10,...,14}
	\draw[black,fill=idealcolor] (2+\x,2) circle (3 pt);
	\foreach \x in {10,...,13}
	\draw[black,fill=idealcolor] (3+\x,1) circle (3 pt);
	\foreach \x in {13,...,16}
	\foreach \y in {3,4}
	\draw[black,fill=idealcolor] (\x,\y) circle (3 pt);
	\foreach \x in {14,...,16}
	\foreach \y in {5,6}
	\draw[black,fill=idealcolor] (\x,\y) circle (3 pt);
	
	\draw[green!70!black,ultra thick] (12,1) circle (10 pt);
	\draw[green!70!black,ultra thick] (11,1) circle (10 pt);
	\draw[green!70!black,ultra thick] (11,0) circle (10 pt);
	\draw[green!70!black,ultra thick] (10,0) circle (10 pt);
	\end{tikzpicture}
        \caption{Standard pairs of $C_{O}= \langle x^2, xy^{2} \rangle$}
        \label{fig:2dnormal_std_primarymax}
    \end{subfigure}
    \caption{A primary decomposition of $\langle x^2y^2, x^3y \rangle$ in $\field[x,xy,xy^2]$}
    \label{fig:2dafffinesemigp_std_primary}
\end{figure*}

\begin{figure*}[t!]
    \centering
    \begin{subfigure}[t]{\threedtwofigure}
        \centering
       \tdplotsetmaincoords{90}{90} 
	\begin{tikzpicture}[tdplot_main_coords, scale=1]
	\tdplotsetrotatedcoords{20}{-10}{30}
	\foreach \x in {0,...,4}
	{
	\draw[step=1cm,gray,very thin, tdplot_rotated_coords] (0,\x,0) -- (0,\x,4);
	\draw[step=1cm,gray,very thin, tdplot_rotated_coords] (0,0,\x) -- (0,4,\x);
	\draw[step=1cm,gray,very thin, tdplot_rotated_coords] (1,\x,0) -- (1,\x,4);
	\draw[step=1cm,gray,very thin, tdplot_rotated_coords] (1,0,\x) -- (1,4,\x);
	\draw[step=1cm,gray,very thin, tdplot_rotated_coords] (2,\x,0) -- (2,\x,4);
	\draw[step=1cm,gray,very thin, tdplot_rotated_coords] (2,0,\x) -- (2,4,\x);
	};
	\foreach \y in {0,...,4}
	\foreach \z in {0,...,4}
	\draw[step=1cm,gray,very thin, tdplot_rotated_coords] (0,\y,\z) -- (2,\y,\z);

	\draw[thick,->,tdplot_rotated_coords] (0,0,0) -- (6,0,0) node[anchor=south west]{$x$}; 
	\draw[thick,->,tdplot_rotated_coords] (0,0,0) -- (0,5.5,0) node[anchor=south west]{$y$}; 
	\draw[thick,->,tdplot_rotated_coords] (0,0,0) -- (0,0,5.5) node[anchor=south]{$z$};
	\foreach \y in {0,...,2}
		\foreach \z in {0,...,2}
		{
			\draw[black,fill=black,tdplot_rotated_coords] (1,2*\y,2*\z) circle (2 pt);
		};
	\foreach \y in {0,...,1}
		\foreach \z in {0,...,1}
			\draw[black,fill=black,tdplot_rotated_coords] (1,1+2*\y,1+2*\z) circle (2 pt);

	\foreach \y in {0,...,1}
		\foreach \z in {0,...,1}
			\draw[black,fill=black,tdplot_rotated_coords] (1,1+2*\y,2+2*\z) circle (2 pt);
	\foreach \y in {0,...,4}
		\foreach \z in {0,...,4}
			\draw[black,fill=black,tdplot_rotated_coords] (2,\y,\z) circle (2 pt);
	\foreach \y in {0,...,1}
	{
		\draw[black,fill=black,tdplot_rotated_coords] (1,1+2*\y,0) circle (2 pt);
	};

	\foreach \y in {0,...,1}
	\foreach \z in {0,...,1}
	{
		\draw[black,fill=white,tdplot_rotated_coords] (0,2*\y+1,2*\z) circle (2 pt);
		\draw[black,fill=white,tdplot_rotated_coords] (0,2*\y+1,1) circle (2 pt);
		\draw[black,fill=white,tdplot_rotated_coords] (0,2*\y,2*\z+1) circle (2 pt);
		\draw[black,fill=white,tdplot_rotated_coords] (0,2*\y+1,3) circle (2 pt);
		\draw[black,fill=white,tdplot_rotated_coords] (0,2*\y+1,4) circle (2 pt);
	};
	\draw[black,fill=white,tdplot_rotated_coords] (0,4,1) circle (2 pt);
	\draw[black,fill=white,tdplot_rotated_coords] (0,4,3) circle (2 pt);

	\foreach \z in {0,...,2}
	\foreach \y in {0,...,2}
	{
		\draw[black,fill=yellow,tdplot_rotated_coords] (0,2*\y,2*\z) circle (2 pt);
	};
	\foreach \z in {0,...,1}
	\foreach \y in {0,...,2}
	{
		\draw[black,fill=red,tdplot_rotated_coords] (1,2*\y,1+2*\z) circle (2 pt);
	};
	\end{tikzpicture}
        \caption{$C_{F}=\langle x,xy,xyz,xyz^{2} \rangle$}
        \label{fig:3dnonnormal_std_primary_022}
    \end{subfigure}
        \begin{subfigure}[t]{\threedtwofigure}
        \centering
       \tdplotsetmaincoords{90}{90} 
	\begin{tikzpicture}[tdplot_main_coords, scale=1]
	\tdplotsetrotatedcoords{20}{-10}{30}
	\foreach \x in {0,...,4}
	{
	\draw[step=1cm,gray,very thin, tdplot_rotated_coords] (0,\x,0) -- (0,\x,4);
	\draw[step=1cm,gray,very thin, tdplot_rotated_coords] (0,0,\x) -- (0,4,\x);
	\draw[step=1cm,gray,very thin, tdplot_rotated_coords] (1,\x,0) -- (1,\x,4);
	\draw[step=1cm,gray,very thin, tdplot_rotated_coords] (1,0,\x) -- (1,4,\x);
	\draw[step=1cm,gray,very thin, tdplot_rotated_coords] (2,\x,0) -- (2,\x,4);
	\draw[step=1cm,gray,very thin, tdplot_rotated_coords] (2,0,\x) -- (2,4,\x);
	};
	\foreach \y in {0,...,4}
	\foreach \z in {0,...,4}
	\draw[step=1cm,gray,very thin, tdplot_rotated_coords] (0,\y,\z) -- (2,\y,\z);

	\draw[thick,->,tdplot_rotated_coords] (0,0,0) -- (6,0,0) node[anchor=south west]{$x$}; 
	\draw[thick,->,tdplot_rotated_coords] (0,0,0) -- (0,5.5,0) node[anchor=south west]{$y$}; 
	\draw[thick,->,tdplot_rotated_coords] (0,0,0) -- (0,0,5.5) node[anchor=south]{$z$};
	\foreach \y in {0,...,2}
		\foreach \z in {0,...,2}
		{
			\draw[black,fill=black,tdplot_rotated_coords] (1,2*\y,2*\z) circle (2 pt);
		};
	\foreach \y in {0,...,1}
		\foreach \z in {0,...,1}
			\draw[black,fill=black,tdplot_rotated_coords] (1,1+2*\y,1+2*\z) circle (2 pt);

	\foreach \y in {0,...,1}
		\foreach \z in {0,...,1}
			\draw[black,fill=black,tdplot_rotated_coords] (1,1+2*\y,2+2*\z) circle (2 pt);
	\foreach \y in {0,...,4}
		\foreach \z in {0,...,4}
			\draw[black,fill=black,tdplot_rotated_coords] (2,\y,\z) circle (2 pt);
	\foreach \z in {0,...,1}
	\foreach \y in {0,...,2}
	{
		\draw[black,fill=black,tdplot_rotated_coords] (1,2*\y,1+2*\z) circle (2 pt);
	};
	\foreach \z in {1,...,2}
	\foreach \y in {0,...,2}
	\draw[black,fill=black,tdplot_rotated_coords] (0,2*\y,2*\z) circle (2 pt);
	\foreach \y in {0,...,1}
	\foreach \z in {0,...,1}
	{
		\draw[black,fill=white,tdplot_rotated_coords] (0,2*\y+1,2*\z) circle (2 pt);
		\draw[black,fill=white,tdplot_rotated_coords] (0,2*\y+1,1) circle (2 pt);
		\draw[black,fill=white,tdplot_rotated_coords] (0,2*\y,2*\z+1) circle (2 pt);
		\draw[black,fill=white,tdplot_rotated_coords] (0,2*\y+1,3) circle (2 pt);
		\draw[black,fill=white,tdplot_rotated_coords] (0,2*\y+1,4) circle (2 pt);
	};
	\draw[black,fill=white,tdplot_rotated_coords] (0,4,1) circle (2 pt);
	\draw[black,fill=white,tdplot_rotated_coords] (0,4,3) circle (2 pt);

	\foreach \y in {0,...,2}
		\draw[black,fill=yellow,tdplot_rotated_coords] (0,2*\y,0) circle (2 pt);
	\foreach \y in {0,...,1}
	{
		\draw[black,fill=blue,tdplot_rotated_coords] (1,1+2*\y,0) circle (2 pt);
	};
	\end{tikzpicture}
        \caption{$C_{G}= \langle x,xz,z^2,xyz \rangle$}
        \label{fig:3dnonnormal_std_primary020}
    \end{subfigure}
    \caption{A primary decomposition of $\langle x,xyz,xyz^{2} \rangle$ in $\field\left[x, x y, x z, x y z, y^{2}, z^{2}\right]$}
    \label{fig:3dafffinesemigp_std_primary}
\end{figure*}

\subsection{Irreducible Decomposition}

We now address the irreducible decomposition of monomial ideals in
semigroup rings using standard pairs. While the existence of monomial irreducible decomposition of monomial ideals in semigroup rings is
known~\cite{MR2110098}*{Corollary~11.5, Proposition~11.41}, an effective combinatorial
description of such a decomposition was missing from the
literature before this work. As a side note, we recall that monomial ideals in semigroup rings can be viewed as
binomial ideals in polynomial rings, and mention that
binomial ideals do not in general have irreducible
decompositions into binomial ideals~\cite{MR3518313}.

In order to decide whether an ideal is irreducible, one must examine
socles. That is the gist of the following result.

\begin{theorem}[\cite{MR1484973}*{Proposition~3.14}]
  \label{thm:irredComponents}
Let $(R,\mathfrak{m})$ be a local noetherian ring and let $M$ be a
finitely generated $R$-module. Let $\mathfrak{p}$ be an associated
prime of $M$, and denote its residue field by $K$. Let $N$ be the
submodule of $M$ whose elements are annihilated by $\mathfrak{p}$.
The number of $\mathfrak{p}$-primary components in an irredundant
irreducible decomposition of the null submodule of $M$ is the
dimension of the localization $N_{\mathfrak{p}}$ as a $K$-vector space.
\end{theorem}

We are now able to determine whether a monomial ideal in $\field[\NN
A]$ is irreducible.

\begin{corollary}
\label{coro:irreducible}
Let $I$ be a $\pf$-primary monomial ideal in $\field[\NN A]$.
The number of overlap classes of standard pairs of $I$ that are
maximal with respect to divisibility equals the number of
components in an irredundant irreducible decomposition
of $I$. In particular, $I$ is irreducible if and only if it has a
single overlap class 
of standard pairs that is maximal with respect to divisibility.
\end{corollary}

\begin{proof}
Theorem~\ref{thm:primaryDecomposition} shows that all
  standard pairs of $I$ belong to $F$. The proof of
  Proposition~\ref{prop:primary} shows that, in this situation, the
  submodule of $\field[\NN A]/I$ whose elements are annihilated by
  $\pf$ is spanned as a $\field$-vector space by the monomials $t^b$
  such that $b \in a+\NN F$ for some standard pair $(a,F)$ whose
  overlap class is maximal with respect to divisibility. After
  localization at $\pf$, this module becomes a vector space over the
  residue field whose dimension equals the number of overlap classes
  of standard pairs that are maximal with respect to
  divisibility. This assertion follows from the following observations:
  $b \in a+\NN F$ and $b' \in a'+\NN F$ 
  where $(a,F)$ and $(a',F)$ are overlapping standard pairs,
  then $b-b' \in \ZZ F$, so that $t^{b-b'}$ is a unit after
  localization at $\pf$. Note also that a linear combination of monomials with
    coefficients in the residue field can only be zero if the pairwise
    differences of the exponents of the monomials belong to $\ZZ F$.
   Now the desired result follows from Theorem~\ref{thm:irredComponents}.  
\end{proof}

By Theorem~\ref{thm:primaryDecomposition}, in order to perform
irreducible decompositions of monomial ideals, it is enough to do it
for primary monomial ideals.

\begin{proposition}
  \label{prop:irreducibleDecomposition}
  Let $I$ be a $\pf$-primary monomial ideal in $\field[\NN A]$, and
  let $[b_1,F],\dots,[b_\ell,F]$ be the maximal overlap classes of
  standard pairs of $I$ with respect to divisibility. For each $1 \leq
  i \leq \ell$, let $T_i=\{c \in
  b+\NN F \mid (b,F) \text{ is a standard pair of } I \text{ whose
    overlap class divides } [b_i,F]\}$. Then $T_i$ is the set of standard
  monomials of a monomial ideal $J_i$. Moreover $J_i \supset I$, $J_i$
  is irreducible, and $I = J_1 \cap \dots \cap J_\ell$ is an
  irredundant irreducible decomposition of $I$.
\end{proposition}

\begin{proof}
  The arguments that proved Proposition~\ref{prop:component} show that
  $J_i$ is a monomial ideal all of whose standard pairs belong to
  $F$. By construction, $[b_i,F]$ is the unique overlap class
  of standard pairs of $J_i$ that is maximal with respect to
  divisibility. It follows that $J_i$ is irreducible by
  Corollary~\ref{coro:irreducible}. The decomposition
  $I=\cap_{i=1}^\ell J_i$ is verified in the same way as the primary
  decomposition in Theorem~\ref{thm:primaryDecomposition}.
\end{proof}

We emphasize that Theorem~\ref{thm:primaryDecomposition} and
Proposition~\ref{prop:irreducibleDecomposition} can be combined to
produce an irredundant irreducible decomposition of a monomial ideal
in $\field[\NN A]$ in terms of its standard pairs. 

\begin{example}
\label{ex:2dirred_decomposition_unique}
The primary decompositions in Example~\ref{ex:ideals3} are also
irredundant irreducible decompositions. We now give two more examples
for non-normal two-dimensional semigroup rings. In the first one, the
primary decomposition of Theorem~\ref{thm:primaryDecomposition} is
already irreducible, in the second one, the primary decomposition is
not irreducible. 
\begin{enumerate}[leftmargin=*]
  \item[(i)] 
Let $A = \big[ \begin{smallmatrix} 1 & 1 & 2 &  3 \\ 1 & 2 &
  0 & 0\end{smallmatrix}\big],$ and consider $I=\langle x^3 y^2 , x^5
y,x^6y \rangle \subset S:=\field[xy,xy^2,x^2,x^3] \cong \field[\NN A]$. The
irreducible decomposition arising from
Proposition~\ref{prop:irreducibleDecomposition} is depicted in
Figure~\ref{fig:2dirred_decomposition_unique}. Blue points in Figure~\ref{fig:2dirred_decomposition_0} including $\{(4,2), (5,3) \}$ are standard monomials of $I$. Whereas, black points in the shaded region are monomials in the ideal. There is no way to generate $(4,2)$ and $(5,3)$ from the generators of $I$ due to the holes $\{(1,0),(2,1)\}$ of $\NN A$. The socle of $\left(\field[\NN A]/J_{1}\right)_{\langle x^2,x^3\rangle}$ is generated by any monomial whose degree is in $(3,3) + \NN (1,2)$, since $[(3,3),\{(1,2) \}]$ is an overlap class which is maximal with respect to divisibility. Similarly, the socle of  $\left(\field[\NN A]/J_{2}\right)_{\langle xy^2\rangle}$ is generated by any monomial whose degree is in $\NN (1,0)$ since $[(0,0),\{(1,0) \}]$ is an overlap class which is maximal with respect to divisibility. Lastly, the socle of $\left(\field[\NN A]/J_{3}\right)$ is $\field\{ x^5y^3\}$ since the overlap class which is maximal with respect to divisibility is $[(5,3),\emptyset] = \{ ((5,3),\emptyset)\}$.
\item[(ii)]
Let $A = \big[\begin{smallmatrix} 2 & 0 & 1  \\ 0 & 1 &
  1 \end{smallmatrix}\big],$ and consider $I = \langle y^2 ,
xy^{2}\rangle \subset \field[x^{2},y,xy] \cong \field[\NN A]$. The
irreducible decomposition of $I$ arising from
Proposition~\ref{prop:irreducibleDecomposition} is depicted in Figure
\ref{fig:2dirred_decomposition_nonunique}. In this case, if $F
=\{(2,0)\}$, the color yellow is used for the standard pair
$((0,0),F)$, the color red for $((1,1),F)$ and blue for
$((0,1),F)$. The pairs $((0,0),F)$ and $((1,1),F)$ are maximal with
respect to divisibility and do not overlap because $(1,0) \notin \NN A$. The socle of $\left(\field[\NN A]/I\right)_{\langle y \rangle}$ is $\field(x^2)\{y,xy\}$ since $[(0,1),F]$ and $[(1,1),F]$ are all overlap classes which are maximal with respect to divisibility.
\end{enumerate}
\end{example}

\begin{figure}[t!]
\begin{subfigure}[t]{\fourfigure}
\[
\begin{tikzpicture}[scale=0.8]

\filldraw[black!20] (4.5,5) --  (3,2) -- (6,2) --  (6,5)--cycle;
\filldraw[black!20] (5.5,2) --  (5,1) -- (6,1) -- (6,2);

 \draw[step=1cm,gray,very thin] (0,0) grid (6,5);
 \draw [<->] (0,5.5) node (yaxis) [above] {$y$}
        |- (6.5,0) node (xaxis) [right] {$x$};
        
\draw[black,fill=stdmcolor] (1,0) circle (3 pt);
\foreach \x in {2,...,6}{
\draw[black,fill=stdmcolor] (\x,0) circle (3 pt);
\draw[black,fill=idealcolor] (6,\x-1)   circle (3pt);
}

\draw[black,fill=stdmcolor] (0,0)   circle (3pt);
\draw[black,fill=stdmcolor] (1,1)   circle (3pt);
\draw[black,fill=stdmcolor] (2,2)   circle (3pt);
\draw[black,fill=stdmcolor] (3,3)   circle (3pt);
\draw[black,fill=stdmcolor] (1,2)   circle (3pt);
\draw[black,fill=stdmcolor] (2,3)   circle (3pt);
\draw[black,fill=stdmcolor] (2,4)   circle (3pt);
\draw[black,fill=stdmcolor] (3,4)   circle (3pt);
\draw[black,fill=stdmcolor] (3,5)   circle (3pt);
\draw[black,fill=stdmcolor] (4,5)   circle (3pt);

\draw[black,fill=stdmcolor] (3,1)   circle (3pt);
\draw[black,fill=stdmcolor] (4,1)   circle (3pt);

\draw[black,fill=stdmcolor] (4,2)   circle (3pt);
\draw[black,fill=stdmcolor] (5,3)   circle (3pt);

\fill[black] (3,2)   circle (3pt);
\fill[black] (5,1)   circle (3pt);
\draw[black,fill=idealcolor] (4,3)   circle (3pt);
\draw[black,fill=idealcolor] (4,4)   circle (3pt);
\draw[black,fill=idealcolor] (5,4)   circle (3pt);
\draw[black,fill=idealcolor] (5,5)   circle (3pt);
\draw[black,fill=idealcolor] (5,2)   circle (3pt);

\draw[ultra thick, red,->] (0,0) -- (2.75,5.5) ;
\draw[ultra thick, red,->] (1,1) -- (3.25,5.5) ;
\draw[ultra thick, red,->] (2,2) -- (3.75,5.5) ;
\draw[ultra thick, red,->] (3,3) -- (4.25,5.5) ;
\draw[ultra thick, violet, ->] (0,0) -- (6.5,0) ;
\draw[green!70!black,ultra thick] (3,1) circle (10 pt);
\draw[green!70!black,ultra thick] (4,1) circle (10 pt);
\draw[green!70!black,ultra thick] (4,2) circle (10 pt);
\draw[green!70!black,ultra thick] (5,3) circle (10 pt);

\fill[white] (1,0)   circle (3pt);
\fill[white] (2,1)   circle (3pt);
\draw (1,0)   circle (3pt);
\draw (2,1)   circle (3pt);

\draw (4,2)   circle (3pt);
\draw (5,3)   circle (3pt);

\end{tikzpicture}
\]
\caption{Standard pairs of $I= \langle x^3 y^2 , x^5 y,x^6y \rangle$   }
\label{fig:2dirred_decomposition_0}
\end{subfigure}
\begin{subfigure}[t]{\fourfigure}
\[
\begin{tikzpicture}[scale=0.8]

\filldraw[black!20] (6,0) --  (2,0) -- (4.5,5) --  (6,5)--cycle;

 \draw[step=1cm,gray,very thin] (0,0) grid (6,5);
 \draw [<->] (0,5.5) node (yaxis) [above] {$y$}
        |- (6.5,0) node (xaxis) [right] {$x$};
        
\draw[black,fill=stdmcolor] (1,0) circle (3 pt);

\draw[black,fill=stdmcolor] (0,0)   circle (3pt);
\draw[black,fill=stdmcolor] (1,1)   circle (3pt);
\draw[black,fill=stdmcolor] (2,2)   circle (3pt);
\draw[black,fill=stdmcolor] (3,3)   circle (3pt);
\draw[black,fill=stdmcolor] (1,2)   circle (3pt);
\draw[black,fill=stdmcolor] (2,3)   circle (3pt);
\draw[black,fill=stdmcolor] (2,4)   circle (3pt);
\draw[black,fill=stdmcolor] (3,4)   circle (3pt);
\draw[black,fill=stdmcolor] (3,5)   circle (3pt);
\draw[black,fill=stdmcolor] (4,5)   circle (3pt);

\fill[black] (3,2)   circle (3pt);
\fill[black] (5,1)   circle (3pt);
\draw[black,fill=idealcolor] (4,3)   circle (3pt);
\draw[black,fill=idealcolor] (4,4)   circle (3pt);
\draw[black,fill=idealcolor] (5,4)   circle (3pt);
\draw[black,fill=idealcolor] (5,5)   circle (3pt);
\draw[black,fill=idealcolor] (5,2)   circle (3pt);
\foreach \x in {2,...,6}{
\draw[black,fill=idealcolor] (\x,0) circle (3 pt);
\draw[black,fill=idealcolor] (6,\x-1)   circle (3pt);
}
\draw[black,fill=idealcolor] (3,1)   circle (3pt);
\draw[black,fill=idealcolor] (4,1)   circle (3pt);

\draw[black,fill=idealcolor] (4,2)   circle (3pt);
\draw[black,fill=idealcolor] (5,3)   circle (3pt);

\draw[ultra thick, red,->] (0,0) -- (2.75,5.5) ;
\draw[ultra thick, red,->] (1,1) -- (3.25,5.5) ;
\draw[ultra thick, red,->] (2,2) -- (3.75,5.5) ;
\draw[ultra thick, red,->] (3,3) -- (4.25,5.5) ;

\fill[white] (1,0)   circle (3pt);
\fill[white] (2,1)   circle (3pt);
\draw (1,0)   circle (3pt);
\draw (2,1)   circle (3pt);

\draw (4,2)   circle (3pt);
\draw (5,3)   circle (3pt);

\end{tikzpicture}
\]
\caption{Standard pairs of $J_{1}= \langle x^{2}, x^{3}\rangle$}
\label{fig:2dirred_decomposition_12}
\end{subfigure}
\begin{subfigure}[t]{\fourfigure}
\[
\begin{tikzpicture}[scale=0.8]

\filldraw[black!20] (2.5,5) --  (1,2) -- (6,2) --  (6,5)--cycle;
\filldraw[black!20] (1.5,2) --  (1,1) -- (6,1) -- (6,2);
\draw[black,->] (0,0) -- (2.75,5.5) ;

 \draw[step=1cm,gray,very thin] (0,0) grid (6,5);
 \draw [<->] (0,5.5) node (yaxis) [above] {$y$}
        |- (6.5,0) node (xaxis) [right] {$x$};
        
\draw[black,fill=stdmcolor] (1,0) circle (3 pt);
\foreach \x in {2,...,6}{
\draw[black,fill=stdmcolor] (\x,0) circle (3 pt);
\draw[black,fill=idealcolor] (6,\x-1)   circle (3pt);
}

\draw[black,fill=stdmcolor] (0,0)   circle (3pt);
\draw[black,fill=idealcolor] (1,1)   circle (3pt);
\draw[black,fill=idealcolor] (2,2)   circle (3pt);
\draw[black,fill=idealcolor] (3,3)   circle (3pt);
\draw[black,fill=idealcolor] (1,2)   circle (3pt);
\draw[black,fill=idealcolor] (2,3)   circle (3pt);
\draw[black,fill=idealcolor] (2,4)   circle (3pt);
\draw[black,fill=idealcolor] (3,4)   circle (3pt);
\draw[black,fill=idealcolor] (3,5)   circle (3pt);
\draw[black,fill=idealcolor] (4,5)   circle (3pt);

\fill[black] (3,2)   circle (3pt);
\fill[black] (5,1)   circle (3pt);
\draw[black,fill=idealcolor] (4,3)   circle (3pt);
\draw[black,fill=idealcolor] (4,4)   circle (3pt);
\draw[black,fill=idealcolor] (5,4)   circle (3pt);
\draw[black,fill=idealcolor] (5,5)   circle (3pt);
\draw[black,fill=idealcolor] (5,2)   circle (3pt);
\draw[black,fill=idealcolor] (3,1)   circle (3pt);
\draw[black,fill=idealcolor] (4,1)   circle (3pt);

\draw[black,fill=idealcolor] (4,2)   circle (3pt);
\draw[black,fill=idealcolor] (5,3)   circle (3pt);

\draw[ultra thick, violet, ->] (0,0) -- (6.5,0) ;

\fill[white] (1,0)   circle (3pt);
\fill[white] (2,1)   circle (3pt);
\draw (1,0)   circle (3pt);
\draw (2,1)   circle (3pt);

\draw (4,2)   circle (3pt);
\draw (5,3)   circle (3pt);

\end{tikzpicture}
\]
\caption{Standard pairs of $J_{2}=\langle xy, xy^{2} \rangle $ }
\label{fig:2dirred_decomposition_20}
\end{subfigure}
\begin{subfigure}[t]{\fourfigure}
\[
\begin{tikzpicture}[scale=0.8]

\filldraw[black!20] (2.5,5) --  (2,4) -- (6,4) --  (6,5)--cycle;
\filldraw[black!20] (4.5,5) --  (3,2) -- (6,2) --  (6,5)--cycle;
\filldraw[black!20] (5,2) --  (4,0) -- (6,0) -- (6,2);

 \draw[step=1cm,gray,very thin] (0,0) grid (6,5);
 \draw [<->] (0,5.5) node (yaxis) [above] {$y$}
        |- (6.5,0) node (xaxis) [right] {$x$};
        
\draw[black,fill=stdmcolor] (1,0) circle (3 pt);
\foreach \x in {2,...,6}{
\draw[black,fill=idealcolor] (6,\x-1)   circle (3pt);
}
\draw[black,fill=stdmcolor] (2,0) circle (3 pt);
\draw[black,fill=stdmcolor] (3,0) circle (3 pt);
\draw[black,fill=stdmcolor] (0,0)   circle (3pt);
\draw[black,fill=stdmcolor] (1,1)   circle (3pt);
\draw[black,fill=stdmcolor] (2,2)   circle (3pt);
\draw[black,fill=stdmcolor] (3,3)   circle (3pt);
\draw[black,fill=stdmcolor] (1,2)   circle (3pt);
\draw[black,fill=stdmcolor] (2,3)   circle (3pt);

\draw[black,fill=stdmcolor] (3,1)   circle (3pt);
\draw[black,fill=stdmcolor] (4,1)   circle (3pt);

\draw[black,fill=stdmcolor] (4,2)   circle (3pt);
\draw[black,fill=stdmcolor] (5,3)   circle (3pt);

\fill[black] (3,2)   circle (3pt);
\fill[black] (5,1)   circle (3pt);
\draw[black,fill=idealcolor] (4,3)   circle (3pt);
\draw[black,fill=idealcolor] (4,4)   circle (3pt);
\draw[black,fill=idealcolor] (5,4)   circle (3pt);
\draw[black,fill=idealcolor] (5,5)   circle (3pt);
\draw[black,fill=idealcolor] (5,2)   circle (3pt);
\draw[black,fill=idealcolor] (2,4)   circle (3pt);
\draw[black,fill=idealcolor] (3,4)   circle (3pt);
\draw[black,fill=idealcolor] (3,5)   circle (3pt);
\draw[black,fill=idealcolor] (4,5)   circle (3pt);

\draw[black,fill=idealcolor] (4,0)   circle (3pt);
\draw[black,fill=idealcolor] (5,0)   circle (3pt);
\draw[black,fill=idealcolor] (6,0)   circle (3pt);

\foreach \x in {0,...,3}
{
\draw[green!70!black,ultra thick] (\x,\x) circle (10 pt);
\draw[green!70!black,ultra thick] (\x+2,\x) circle (10 pt);
};
\draw[green!70!black,ultra thick] (4,1) circle (10 pt);
\draw[green!70!black,ultra thick] (3,0) circle (10 pt);
\draw[green!70!black,ultra thick] (1,2) circle (10 pt);
\draw[green!70!black,ultra thick] (2,3) circle (10 pt);

\fill[white] (1,0)   circle (3pt);
\fill[white] (2,1)   circle (3pt);
\draw (1,0)   circle (3pt);
\draw (2,1)   circle (3pt);

\draw (4,2)   circle (3pt);
\draw (5,3)   circle (3pt);

\end{tikzpicture}
\]
\caption{Standard pairs of $J_{3}=\langle x^{4}, x^{3}y^{2}, x^{2}y^{4},x^3y^4,x^5\rangle$ }
\label{fig:2dirred_decomposition_00}
\end{subfigure}
\caption{An irredundant irreducible decomposition of $I=J_{1} \cap J_{2} \cap J_{3}$ in $\field[xy,xy^2,x^2,x^3]$.}
\label{fig:2dirred_decomposition_unique}
\end{figure}
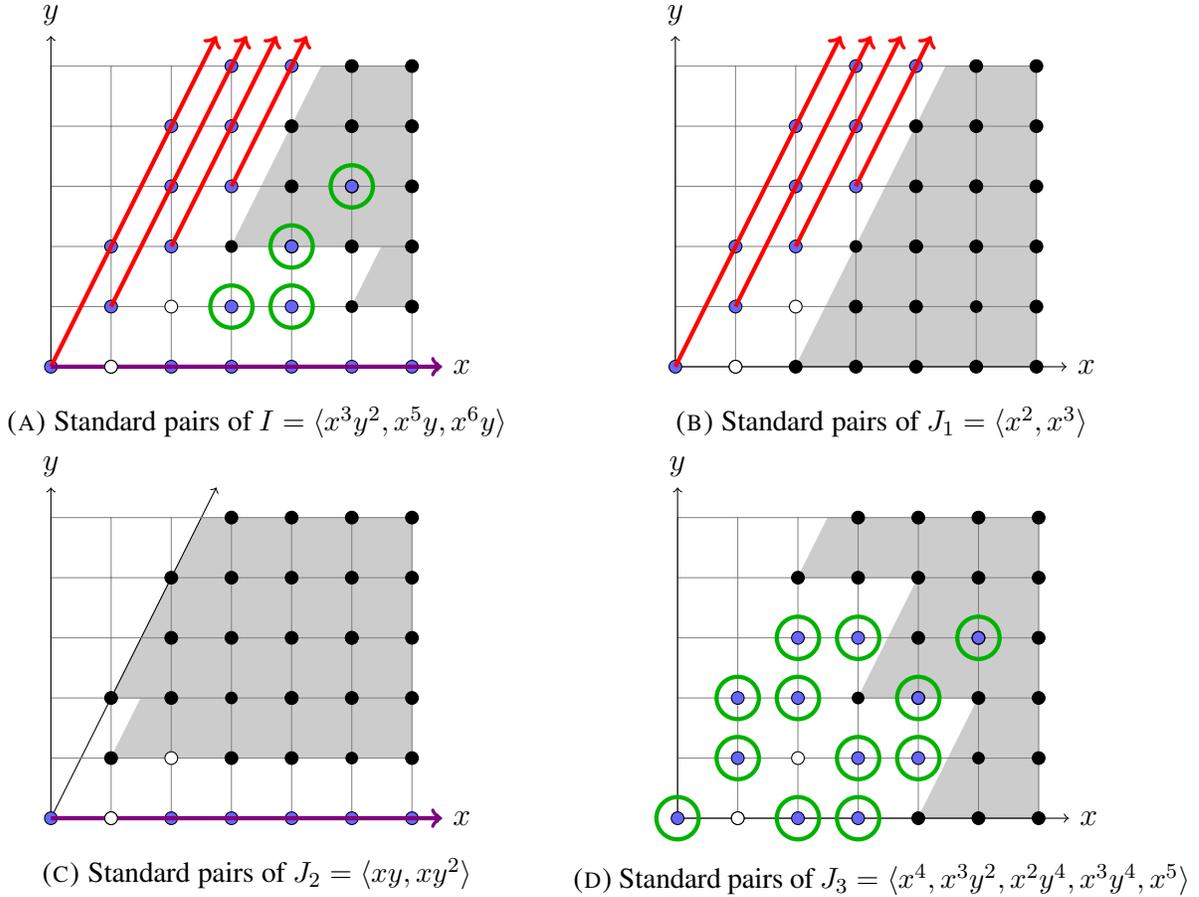

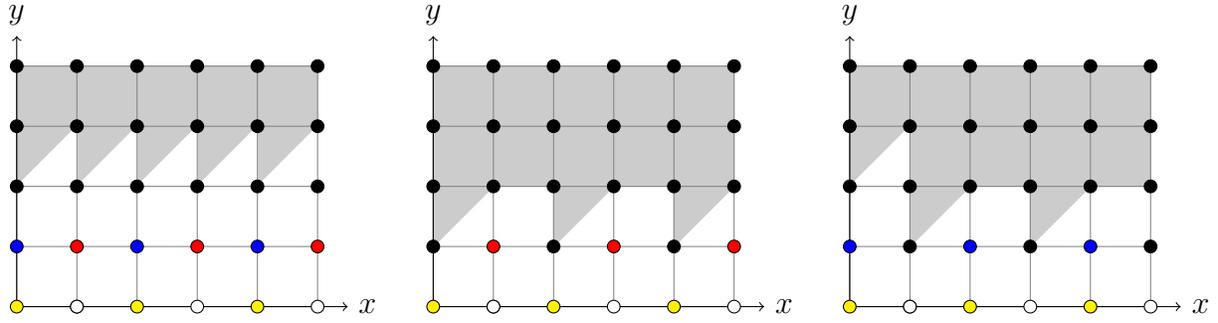
\begin{figure}[t!]
\begin{subfigure}[t]{\threefigure}
\[
\begin{tikzpicture}[scale=0.8]

\filldraw[black!20] (0,2) --  (0,4) -- (5,4) --  (5,2)--cycle;
\foreach \x in {0,...,4}
\filldraw[white] (\x,2) --  (\x+1,3) -- (\x+1,2) --  cycle;

 \draw[step=1cm,gray,very thin] (0,0) grid (5,4);
 \draw [<->] (0,4.5) node (yaxis) [above] {$y$}
        |- (5.5,0) node (xaxis) [right] {$x$};
        
\draw[black,fill=stdmcolor] (1,0) circle (3 pt);
\foreach \x in {0,...,2}{
\draw[black,fill=yellow] (2*\x,0) circle (3 pt);
\draw[black,fill=blue] (2*\x,1) circle (3 pt);
\draw[black,fill=red] (2*\x+1,1) circle (3 pt);
}

\foreach \x in {0,...,5}
\foreach \y in {2,...,4}
\draw[black,fill=idealcolor] (\x,\y)   circle (3pt);

\foreach \x in {1,3,5}
\draw[black,fill=white] (\x,0) circle (3 pt);
\end{tikzpicture}
\]
\caption{Standard pairs of $I= \langle y^2 , xy^2 \rangle$   }
\label{fig:2dirred_decomposition_n_0}
\end{subfigure}
\begin{subfigure}[t]{\threefigure}
\[
\begin{tikzpicture}[scale=0.8]

\filldraw[black!20] (0,2) --  (0,4) -- (5,4) --  (5,2)--cycle;
\filldraw[black!20] (0,1) --  (0,2) -- (1,2) --cycle;
\filldraw[black!20] (2,1) --  (2,2) -- (3,2) --cycle;
\filldraw[black!20] (4,1) --  (4,2) -- (5,2) --cycle;

 \draw[step=1cm,gray,very thin] (0,0) grid (5,4);
 \draw [<->] (0,4.5) node (yaxis) [above] {$y$}
        |- (5.5,0) node (xaxis) [right] {$x$};
        
\draw[black,fill=stdmcolor] (1,0) circle (3 pt);
\foreach \x in {0,...,2}{
\draw[black,fill=yellow] (2*\x,0) circle (3 pt);
\draw[black,fill=black] (2*\x,1) circle (3 pt);
\draw[black,fill=red] (2*\x+1,1) circle (3 pt);
}

\foreach \x in {0,...,5}
\foreach \y in {2,...,4}
\draw[black,fill=idealcolor] (\x,\y)   circle (3pt);

\foreach \x in {1,3,5}
\draw[black,fill=white] (\x,0) circle (3 pt);
\end{tikzpicture}
\]
\caption{Standard pairs of $J_{1}= \langle y\rangle$}
\label{fig:2dirred_decomposition_n_1}
\end{subfigure}
\begin{subfigure}[t]{\threefigure}
\[
\begin{tikzpicture}[scale=0.8]

\filldraw[black!20] (0,2) --  (0,4) -- (5,4) --  (5,2)--cycle;
\filldraw[black!20] (1,1) --  (1,2) -- (2,2) --cycle;
\filldraw[black!20] (3,1) --  (3,2) -- (4,2) --cycle;

\filldraw[white] (0,2) --  (1,3) -- (1,2) --  cycle;

 \draw[step=1cm,gray,very thin] (0,0) grid (5,4);
 \draw [<->] (0,4.5) node (yaxis) [above] {$y$}
        |- (5.5,0) node (xaxis) [right] {$x$};
        
\draw[black,fill=stdmcolor] (1,0) circle (3 pt);
\foreach \x in {0,...,2}{
\draw[black,fill=yellow] (2*\x,0) circle (3 pt);
\draw[black,fill=blue] (2*\x,1) circle (3 pt);
\draw[black,fill=black] (2*\x+1,1) circle (3 pt);
}

\foreach \x in {0,...,5}
\foreach \y in {2,...,4}
\draw[black,fill=idealcolor] (\x,\y)   circle (3pt);

\foreach \x in {1,3,5}
\draw[black,fill=white] (\x,0) circle (3 pt);
\end{tikzpicture}
\]
\caption{Standard pairs of $J_{2}=\langle xy,y^2 \rangle $ }
\label{fig:2dirred_decomposition_n_2}
\end{subfigure}
\caption{An irredundant irreducible decomposition of $I=J_{1} \cap J_{2}$ in $\field[x^2,y,xy]$.}
\label{fig:2dirred_decomposition_nonunique}
\end{figure}

\subsection{Multiplicities and Counting Standard Pairs}

One of the main goals of this article is to provide an effective
computation of irreducible and primary decomposition of monomial
ideals in semigroup rings. Given our previous results, this can be
achieved once we know how to compute standard pairs. A key question 
then is whether a monomial ideal $I$ always has finitely many standard
pairs. We answer this question in the affirmative, by linking the
number of (overlap classes) of standard pairs to the multiplicities of
associated primes introduced in~\cite{MR1339920}, which we now recall.

Let $I$ be a monomial ideal in $\field[\NN A]$, and let $\pf$ be an
associated prime of $I$. Following~\cite{MR1339920}, we define
$\mult_I(\pf)$ to be the length of a maximal strictly increasing chain
of ideals
\begin{equation}
  \label{eqn:chain}
I = J_1 \subsetneq J_2 \subsetneq J_3 \subsetneq \cdots J_\ell
\subsetneq J
\end{equation}
where $J = \cup_{j>0} (I:\pf^j)$ is the intersection of the primary
components of $I$ with associated primes not containing $\pf$, and
each $J_k$ is the intersection of $J$ with some $\pf$-primary ideal. 
Equivalently, $\mult_I(\pf)$ is the length of the largest ideal of finite length
in $\field[\NN A]_{\pf} / I \field[\NN A]_{\pf}$. We emphasize that
$\mult_I(\pf)$ is finite.

The following statement generalizes~\cite{MR1339920}*{Lemma~3.3}. Our
argument here is based on the proof of that result.

\begin{proposition}
  \label{prop:multiplicity}
  Let $I$ be a monomial ideal in $\field[\NN A]$ and let $\pf$ be an
  associated prime of $I$. Then $\mult_I(\pf)$
  equals the number of overlap classes of standard pairs of $I$ that
  belong to $F$. 
\end{proposition}

\begin{proof}
  Recall the decomposition $I = \cap_{G \in \mathscr{S}} C_G$
  from Theorem~\ref{thm:primaryDecomposition}. Let $\bar{I} = \cap_{G
    \in \mathscr{S}, G \supset F} C_G$. Then $\field[\NN A]_{\pf} / I
  \field[\NN A]_{\pf}$ is isomorphic to $\field[\NN A]_{\pf} /
  \bar{I} \field[\NN A]_{\pf}$, so that
  $\mult_I(\pf)=\mult_{\bar{I}}(\pf)$. Moreover, the standard pairs of
  $C_G$ are (essentially by construction) the standard pairs of $I$
  belonging to $G$. This implies that $I$ and $\bar{I}$ have the same
  standard pairs belonging to $F$. We have thus reduced the proof to the case
  when $I = \bar{I}$, and we assume this henceforth. In particular, $J
  = \cap_{G \in \mathscr{S}, G \neq F} C_G$ is the ideal used
  in~\eqref{eqn:chain} in this case.
 
  Set $J_1=I$. Let $[a,F]$ be an overlap
  class of standard pairs of $I$ which is maximal with respect to
  divisibility. Let 
\[
  T_2 = \bigcup_{
    \substack{
    (b,G) \in \stdPairs(I) \text{ with } \\
     G \supseteq F \text{
      and } (b,G) \notin [a,F]}} \bigg( b+ \NN G \bigg) \quad
\text{and} \quad 
  S_2 = \bigcup_{\substack{(b,F) \in \stdPairs(C_F) \\ (b,F) \notin [a,F]}}
  \bigg(b+ \NN F\bigg).
\]
    Then $T_2$ is the set of
    standard monomials of an ideal $J_2 \supsetneq I$, $S_2$ is the
    set of standard monomials of a $\pf$-primary ideal $C_2$ and $J_2
    = J \cap C_2$.
    
    To see that $T_2$ is the set of standard monomials of a monomial
    ideal $J_2$, let $b \notin T_2$ and $c \in \NN A$. If $b+c \in T_2$,
    then $b+c \in a'+\NN G$ for some standard pair $(a',G)$ of $I$
    (not belonging to $[a,F]$),
    and therefore $t^{b+c} \notin I$. We
    conclude that $t^b \notin I$, which implies that $b \in \hat{a} +
    \NN F$, where $(\hat{a},F)$ is a standard pair of $I$ belonging to
    $[a,F]$. 
    By assumption on $I$, all of its standard pairs belong to faces
    containing $F$. It follows that $(\hat{a},F)$ divides $(a',G)$,
    which implies that $G=F.$ Hence, $[a,F]$ is not maximal with respect
    to divisibility, a contradiction. 

    The same argument shows that $S_2$
    is the set of standard monomials of a monomial ideal $C_2$, which is
    primary by Proposition~\ref{prop:primary}. The equality 
    $J_2=J\cap C_2$ holds by construction. The localization
    of $J_2/I$ at $\pf$ is a one-dimensional vector space over the
    residue field since monomials in overlapping standard pairs
    differ by a unit in the localization.

    Applying this argument successively, one constructs a chain as
    in~\eqref{eqn:chain}, whose length is the number of overlap
    classes of standard pairs of $I$ belonging to $F$. This chain is
    maximal since the successive quotients are one-dimensional after
    localization at $\pf$.
\end{proof}

\begin{corollary}
  \label{coro:finitelyManyOverlapClasses}
  Let $I$ be a monomial ideal in $\field[\NN A]$. There are finitely
  many overlap classes of standard pairs of $I$.
\end{corollary}

\begin{proof}
  By Proposition~\ref{prop:multiplicity}, the number of overlap classes of
  standard pairs of $I$ is the sum of the multiplicities of its
  associated primes. Since these multiplicities are finite, and $I$
  has only finitely many associated primes, the desired result
  follows.
\end{proof}

The following is the main result of this section.

\begin{theorem}
  \label{thm:finitelyManyStdP}
  Let $I$ be a monomial ideal in $\field[\NN A]$. Then $I$ has
  finitely many standard pairs.
\end{theorem}

In order to prove Theorem~\ref{thm:finitelyManyStdP}, we need an
auxiliary result.

\begin{lemma}
\label{lemma:stdPairMinimality}
Let $I$ be a monomial ideal in $\field[\NN A]$. If $(a,F)$ is a
standard pair of $I$, then $t^a$ is minimal with respect to
divisibility in $\{ t^b \mid b \in (a + \RR F) \cap \NN A\}$. 
\end{lemma}

\begin{proof}
Let $a' \in (a + \RR F)\cap \NN A$ and assume that
$t^{a'}$ divides $t^a$. Our goal is to show that $a'=a$.

We claim that $(a',F)$ is a proper pair of $I$. 
To see this, let $c \in \NN F$. If $a'+c \notin \stdMono(I)$, then $t^{a'+c}
\in I$. 
Since $t^{a'}$ divides $t^a$, it follows that $t^{a+c} \in I$, so that
$a+c \notin \stdMono(I)$. 
But this contradicts the fact that $(a,F)$ is proper, and the
claim follows. Moreover, since $t^{a'}$ divides
$t^a$, we have that $a'+ \NN F \supseteq a + \NN F$, in other words,
$(a', F) \succ (a,F)$, and as $(a,F)$ is a standard pair, we see that 
$a'+ \NN F = a + \NN F$, and so $a'-a \in \NN F$ and also $a - a' \in
\NN F$. 
Hence $t^{a-a'}$ is a unit in $\field[\NN
A]$. By the strong convexity assumption, the only units in $\field[\NN
A]$ belong to $\field$, from which we conclude that $a=a'$.
\end{proof}

\begin{proof}[Proof of Theorem~\ref{thm:finitelyManyStdP}]
  By Corollary~\ref{coro:finitelyManyOverlapClasses}, it is enough to
  show that the equivalence classes under the overlap relation are
  finite. 

  Let $(a,F)$ and $(a',F)$ be overlapping standard pairs of $I$. In
  this case, $a-a' \in \ZZ F$, so that $a + \RR F = a' +\RR F$. By
  Lemma~\ref{lemma:stdPairMinimality}, this implies that $a$ and $a'$
  are minimal with respect to divisibility in $(a+ \RR F) \cap \NN
  A$. It follows that $a-a' \notin
  \NN A$ and $a'-a\notin \NN A$. Now
  apply Lemma~\ref{lemma:Dickson}.
\end{proof}

\section{Algorithms for finding and using standard pairs}
\label{sec:algorithms}

We now turn to concrete methods to compute standard pairs and use
standard pairs to produce primary and irreducible decompositions for
monomial ideals in an affine semigroup ring. 

The algorithms outlined in this article are based on three
important facts.

\begin{enumerate}[leftmargin=*]
\item \label{item:faceLattice}
The complete face lattice of the cone $\RR_{\geq 0} A$ can be computed if $A$ is given. This includes finding the primitive integral support functions (Definition~\ref{def:integerDistance}) for all the facets of $\RR_{\geq 0} A$.
\item \label{item:ILP}
  A (homogeneous or inhomogeneous) system of linear equations and
  inequalities with integer coefficients can be solved, in the sense
  that there exist algorithms to find the coordinatewise minimal
  solutions and free variables.
\item \label{item:stdPairsOverPr}
  There are algorithms to compute standard pairs for monomial ideals
  in polynomial rings. 
\end{enumerate}

We emphasize that solving linear systems over the integers is a
fundamental problem in many areas and continues to be the focus of much research, especially in convex and discrete optimization; finding the faces of a cone is an important basic question in discrete geometry. There are many
approaches to carry out the computational tasks mentioned above.
We discuss specific implementations in Section~\ref{sec:implementation}.

Relevant questions that can be easily stated as systems of
linear equations and inequalities include the following. Given $a \in \ZZ^d$, and
$F$ a finite subset of $\ZZ^d$. Does $a$ belong to $\ZZ F$? Does $a$
belong to $\NN F$? With these in hand and knowledge of the faces of
$\RR_{\geq 0} A$, we can determine, given two pairs $(a,F)$ and
$(b,G)$ of $A$, whether $(a,F) \prec (b,G)$, whether $(a,F)$ divides
$(b,G)$, and whether $(a,F)$ and $(b,F)$ overlap.

In what follows, for $F$ a face of $A$, we use $\NN^F$ to denote
$\NN^{|F|}$ with coordinates indexed by the elements of $F$. If $G$ is
another face of $A$, and $F \subset G$, then we consider the natural inclusion $\NN^F
\subset \NN^G$ where elements of $\NN^F$ are considered as elements of
$\NN^G$ whose coordinates indexed by $G\minus
F$ are zero.

The following algorithm is the main building block for computing standard
pairs in Theorem~\ref{thm:computeStdPairs}. This algorithm decomposes the difference between translations of faces as a finite union of such translations. Its proof is inspired by ideas from~\cite{HTYcomputingHoles}. 

\begin{theorem}
  \label{thm:pairDifference}
  Let $b, b' \in \NN A$ and let $G, G'$ be faces of $A$ such that $G \subseteq G'$.
The procedure below generates a set $\cup_{(a,F) \in C} (a+\NN F)$ equal to $(b+\NN G) \minus (b'+\NN G')$.
\begin{enumerate}[leftmargin=*]
\item Find the set $B=\{(u,w) \in \NN^G \times \NN^{G'} \mid b+G\cdot u = b' + G'\cdot w \}.$
\item Collect all the minimal generators of the projection of $B$ onto the first component.
\item Construct the ideal $J$ generated by the collection of minimal generators found in the previous step.
\item Find the standard pairs of $J$.
\item Return $\cup_{(u,\sigma) \in \stdPairs(J)} (b+G\cdot u + \NN \{a_i \mid i \in \sigma\})$.
\end{enumerate}
\end{theorem}

\begin{proof}
  Consider the set 
  \begin{equation}
    \label{eqn:intersection}
  \{ u \in \NN^G \mid b + G \cdot u \in (b'+\NN G') \}.
    \end{equation}
  Since $G \subseteq G'$, this 
  is the set of the exponents of the monomials in a monomial ideal $J$ in $\field[\NN^G] =
  \field[y_j \mid a_i \in G]$. Observe that
  \begin{align}
    (b+\NN G) \minus (b'+\NN G') & = \{b + G \cdot v \mid v \in \NN^G
                                   \text{ does not belong to the
                                   set}~\eqref{eqn:intersection} \} \nonumber
                                   \\
                                 & = \{b + G \cdot v \mid y^v \notin J \}   \label{eqn:difference}
  \end{align}
  Our goal is thus to find the standard monomials of $J$. First we
  determine minimal generators for $J$, which are the coordinatewise
  minimal elements of~\eqref{eqn:intersection}.

  Now consider
  \begin{equation}
    \label{eqn:upstairs}
  \{(u,w) \in \NN^G \times \NN^{G'} \mid b+G\cdot u = b' + G'\cdot w \}.
  \end{equation}
 The set~\eqref{eqn:intersection} is the projection onto the
  first component of the set~\eqref{eqn:upstairs}. 
  
  Let $\bar{u}$ be a coordinatewise minimal element
  of~\eqref{eqn:intersection}.
  Then there is $\bar{w} \in \NN^{G'}$ such that
  $(\bar{u},\bar{w})$ belongs to~\eqref{eqn:upstairs}. Let $(u,w)$ be a
  coordinatewise minimal element of~\eqref{eqn:upstairs} that is 
  coordinatewise less than or equal to $(\bar{u},\bar{w})$. It follows
  that $u$
  belongs to~\eqref{eqn:intersection} and is coordinatewise less than
  or equal to $\bar{u}$ so that
  $u=\bar{u}$. This shows that the coordinatewise minimal
  elements of~\eqref{eqn:intersection} are the projections of the coordinatewise
  minimal elements of~\eqref{eqn:upstairs}. Since the
  set~\eqref{eqn:upstairs} is the set of integer solutions of a system
  of linear equations and inequalities defined over $\ZZ$, its
  coordinatewise minimal elements can be computed. That there are
  finitely many such elements follows from Dickson's Lemma.

  Since we now know the generators of the monomial ideal $J$, we can compute its standard
  pairs and write 
  \[
    (b+\NN G) \minus (b' + \NN G') = \cup_{(u,\sigma) \in
      \stdPairs(J)} (b+G\cdot u + \NN \{a_i \mid i \in \sigma\})
    \]
  We use the convention that the standard pairs of $J \subset
  \field[\NN^G]$ are of the form $(u,\sigma)$ where $u \in \NN^G$ and
  $\sigma \subset \{ i \mid a_i \in G \}$. By definition, the fact
  that $(u,\sigma)$ is a standard pair of $J$ implies that 
  $y^u \prod_{i\in \sigma} y_i^{\lambda_i} \notin J$ for all $\lambda_i
  \in \NN$, $i\in \sigma$.

  It only remains to be proven that if $(u,\sigma)$ is
  a standard pair of $J$ then $\{a_i \mid i\in \sigma\}$ is a face of
  $G$. 

  Let $(u,\sigma)$ be a standard pair of $J$, and let $F$ be the
  smallest face of $G$ such that $\NN \{a_i \mid i \in \sigma\}$ meets
  the relative interior of $\RR_{\geq 0}F$. Let $\sum_{i\in
    \sigma} \lambda_i a_i$ be an element of the relative interior of
  $F$ with $\lambda_i \in \NN$ for $i\in \sigma$, and set $\lambda \in
  \NN^G$ whose $i$th coordinate is $\lambda_i$ if $i\in \sigma$ and
  $0$ otherwise. Then $y^{u+N \lambda} \notin J$ for all $N
  \in \NN$. Now let $a =\sum_{a_i \in F} \mu_i a_i\in \NN F$, with the
  $\mu_i\in \NN$, and let $\mu \in \NN^G$ whose $i$th coordinate is
  $\mu_i$ if $a_i \in F$ and $0$ otherwise, so that $a=G\cdot \mu$.
  Since $\sum_{i\in \sigma} \lambda_i a_i$ is in the relative interior of $\RR_{\geq
    0}F$,  we may choose $N$ large enough that $N G \cdot \lambda -a
  \in \NN F$, and we may write $G \cdot (N \lambda - \mu) = G \cdot \nu$ with $\nu \in \NN^F \subset
  \NN^G$. But then $G\cdot (\nu+\mu) = G\cdot (N\lambda)$, and as
  $b+ G\cdot(u+N \lambda) \notin b'+\NN G'$ (because $y^{u+N\lambda} \notin
  J$), we have $b +G\cdot(u+\nu+\mu) \notin b'+\NN G'$, which in turn
  implies that $y^{u+\mu} \notin J$. It follows that
  $(u,\{i\mid a_i \in F\})$ is a proper pair of $J$. Since
  $(u,\sigma)$ is a standard pair of $J$, we conclude that $\sigma =
  \{i\mid a_i\in F\}$. 
\end{proof}

We need two more auxiliary results for the computation of standard
pairs in Theorem~\ref{thm:computeStdPairs}. Here is the first one.

\begin{lemma}
  \label{lemma:minimalElements}
  Let $F$ be a face of $A$ and let $a\in \NN A$. The output of the algorithm below is the minimal elements (with respect to
  divisibility) of the set $(a + \RR F) \cap \NN A$.
\begin{enumerate}[leftmargin=*]
\item Find $\{ u \in \NN^n \mid \varphi_H(A \cdot u) = \varphi_H(a) \text{ for all } H
    \text{ facet of } A, H \supseteq F \}$ where $\varphi_{H}$ is the primitive integral support function
  of the facet $H$ of $A$.
\end{enumerate}
\end{lemma}

\begin{proof}
 The primitive integral support functions
  of the facets of $A$ (Definition~\ref{def:integerDistance})
  are linear forms with integer coefficients, and which can be computed.
  Then $b \in (a + \RR F)$ if and only if $\varphi_H(b) = \varphi_H(a)$ for all facets $H$
  of $A$ containing $F$.
  The elements of $(a + \RR F) \cap \NN A$ that are minimal with
  respect to divisibility are the elements of the form $A \cdot u$,
  where $u$ is a coordinatewise minimal element of:
  \begin{equation}
    \label{eqn:translate}
    \{ u \in \NN^n \mid \varphi_H(A \cdot u) = \varphi_H(a) \text{ for all } H
    \text{ facet of } A, H \supseteq F \}.
  \end{equation}
  This set is given by integer linear equations and inequalities, and its
  coordinatewise minimal elements can be computed.
\end{proof}

\begin{definition}
  \label{def:cover}
  Let $I$ be a monomial ideal in $\field[\NN A]$ whose set of standard monomials is $\stdMono(I)$. A \emph{cover} of $\stdMono(I)$ is a finite collection $C$ of pairs of $A$ such that
  \[
    \stdMono(I) = \cup_{(a,F)\in C} (a+\NN F).
  \]
\end{definition}

\begin{proposition}
  \label{prop:refineCover}
  Let $I$ be a monomial ideal in $\field[\NN A]$. The algorithm below has a cover of standard monomials of $I$ as its input. Its output is the set of standard pairs of $I$.
\begin{enumerate}[leftmargin=*]
\item For each $(a,F) \in C$, find $\{ (b,F) : b \in  (a + \RR F) \cap \NN A\}$ using Lemma~\ref{lemma:minimalElements}.
\item Let $C_{1}$ be the union all the sets from the previous step.
\item For each $(a,F) \in C_{1}$ and $G$ a face that is not strictly contained in $F$,
\begin{enumerate}[leftmargin=*]
\item If there is a pair $(b,G') \in C_{1}$ for $G' \supseteq G$,
\begin{enumerate}[leftmargin=*]
\item If $(a+\NN G) \minus (b + \NN G') \subsetneq a + \NN G$
\begin{enumerate}[leftmargin=*]
\item Delete $(a,F)$ in $C_{1}$ and append all maximal pairs of the decomposition of $(a+\NN G) \minus (b + \NN G')$ by Theorem~\ref{thm:pairDifference}.
\end{enumerate}
\end{enumerate}
\end{enumerate}
\item Suppose $C_{2}$ is the changed collection from $C_{1}$ by the previous step. 
\item If $C_{2} \neq C$, set $C:= C_{2}$ and go to (1). Otherwise, return $C_{2}$.
\end{enumerate}
\end{proposition}

\begin{proof}
  Let $C_0$ be a cover of the standard monomials of $I$. Then all the
  elements of $C_0$ are proper pairs of $I$. 
  For $(a,F) \in C$, if $b \in \NN A$ divides $a$, then $(b,
  F)$ is also a proper pair of $I$.
  
  For each $(a,F) \in C_0$, use Lemma~\ref{lemma:minimalElements} to
  compute the minimal elements with respect to divisibility of $(a+\RR F) \cap \NN A$, and replace $(a,F)$ by the
  collection of pairs $(b,F)$, where $b$ is a minimal element of 
  $(a+\RR F) \cap \NN A$ that divides $a$. In this way we obtain
  another collection of pairs $C_1$, which is also a cover of the standard
  monomials of $I$.

  Next, given $(a,F)$ in $C_1$, and $G$ a face of $A$ that is not strictly contained
  in $F$, we can determine algorithmically whether $(a,G)$ is
  a proper pair of $I$, as follows. First, if $C_1$ contains no pairs
  of the form $(b,G') \in C_1$ with $G'
  \supseteq G$, then $(a,G)$ is not proper. Otherwise, find whether
  there is a pair  $(b,G') \in C_1$ with $G'
  \supseteq G$ such that 
  $(a+\NN G) \minus (b + \NN
  G') \subsetneq a + \NN G$. If no such pair exists, $(a,G)$ is not
  proper. This is because, when $(a,G)$ is proper, the elements of $a+\NN G$ are exponents of standard
    monomials. Since $C_1$ is a cover of standard monomials, $a+ \NN G
  = (a + \NN G)  \cap (\cup_{(b,F)\in C_1} (b+\NN F)) = \cup_{(b,F)\in
    C_1}\big( (a+\NN G) \cap (b+\NN F) )$. Each intersection $(a+\NN
  G) \cap (b+\NN F)$ is a finite union of sets $c+\NN F'$ where $F'
  \subset G\cap F \subset G$. If all the intersections involve faces
  that are strictly contained in $G$, then we have written $a+\NN G$ 
  as a finite union of sets $c+\NN F'$ with $F'$ strictly
  contained in $G$, which is impossible for dimension reasons.

  If such a pair exists, $(a+\NN G) \minus (b + \NN
  G')$ is a union of sets of the form $a'+\NN G''$ where $G''$ is a
  proper face of $G$, so we reduce to verifying whether the pairs
  $(a',G'')$ in the union are proper pairs of $I$. This yields an
  iterative procedure to determine whether $(a,G)$ is proper. 

  If $(a,F) \in C_1$, replace $(a,F)$ by all pairs of the form
  $(a,G)$, where $G$ is not strictly contained in $F$, $(a,G)$ is proper for
  $I$, and $G$ is maximal with this property. We obtain a finite collection of pairs $C_2$, which is still a
  cover for the standard monomials of $I$. Now apply to $C_2$ the same
  procedure we used to go from $C_0$ to $C_1$ to construct a new cover
  $C_3$, and apply to $C_3$ the same procedure we applied to $C_1$, to
  get a new cover $C_4$.
  
  We claim that repeating this process yields, after finitely many
  iterations, a cover $C$ which is stable under the given operations.  
  To see this, first, our procedure 
  replaces a proper pair by a collection of proper pairs, all of which
  are greater than or equal to the original pair with respect to the
  partial order $\prec$ (see
  Definitions~\ref{def:Pairs} and~\ref{def:standardPairs}). Next, the set of proper pairs of $I$ has finitely many elements that are maximal with respect to $\prec$,
  namely the standard pairs (Theorem~\ref{thm:finitelyManyStdP}). Finally, 
  $\prec$-chains that are bounded above are finite, because of the
  strong convexity assumption on $\RR_{\geq 0}A$. From these
  observations it follows that 
  our procedure arrives at a stable cover $C$ after finitely many steps.
  
  The stable cover $C$ has the
  following properties
  \begin{itemize}[leftmargin=*]
    \item If
  $(a,F) \in C$ and $F'$ is a face of $A$ that strictly contains $F$,
  then $(a,F')$ is not a proper pair of $I$.
    \item If $(a,F) \in C$ and $(a,G)$ is a proper pair of $I$, then
      $C$ contains a pair $(a,G')$ such that $G'\supseteq G$.
    \end{itemize}
    
  We claim that $C$ contains the standard pairs of $I$.
  
  Let $(b, G)$ be a standard pair of $I$, and let $(a,F) \in C$ such that $b \in
  a+\NN F$. Every element of $a+\NN G$ divides some element of $b+\NN
  G$. This implies that $(a,G)$ is a proper pair of $I$ since
  $(b,G)$ is proper. By construction of $C$, $C$ contains a pair
  $(a,F')$ such that $F'$ contains $G$. Then $b+\NN G \subset a+\NN
  F'$. Since $(a,F')$ is proper, and $(b,G)$ is standard, we must have
  $(a,F')=(b,G)$, which means that $(b,G) \in C$.

  Thus, in order to obtain the standard pairs of $I$, we select the
  elements of $C$ that are maximal with respect to $\prec$.
\end{proof}

We are now ready to compute standard pairs.

\begin{theorem}
  \label{thm:computeStdPairs}
The algorithm below has the set of (monomial) generators of a
  monomial ideal $I = \< t^{b_1},t^{b_2}, \cdots, t^{b^{n}}\>$ in $\field[\NN A]$ as its input. Its output is the set
  of standard pairs of $I$.
\begin{enumerate}[leftmargin=*]
\item Find a cover $C$ of $I_{1}:=\< t^{b_1}\>$ by refining the cover, obtained from the decomposition of $\NN A  \minus (b+\NN A)$ using Theorem~\ref{thm:pairDifference} and Proposition~\ref{prop:refineCover}.
\item If $n=1$, return the cover. Otherwise, let $C$ be the cover found by the previous step. Set $i=2$.
\begin{enumerate}[leftmargin=*]
\item For $(a,F) \in C$, 
\begin{enumerate}[leftmargin=*]
\item Compute the decomposition $(a+\NN F ) \minus (b_i + \NN A)$ from Theorem~\ref{thm:pairDifference} and save them on $C'$.
\end{enumerate}
\item Refine $C'$ using Proposition~\ref{prop:refineCover}. If $i=n$, return the refined $C'$. Otherwise, set $C:=C'$ and go to (a) with increased $i$ by $1$.
\end{enumerate}
\end{enumerate}
  
\end{theorem}

\begin{proof}
  Suppose that $I = \< t^b \>$. Then the set of standard monomials of
  $I$ is $\NN A  \minus (b+\NN A)$, and we can compute the standard
  pairs of $I$ using Theorem~\ref{thm:pairDifference} and
  Proposition~\ref{prop:refineCover}.

  If $I = \< t^{b_1},t^{b_2}\>$, first compute the standard pairs of
  $\<t^{b_1}\>$ and for each such standard pair $(a,F)$, compute
  $(a+\NN F ) \minus (b_2 + \NN A)$. This yields a cover of the
  standard monomials of $I$, which can be massaged using
  Proposition~\ref{prop:refineCover} to obtain the standard pairs of
  $I$.

  In general, if the standard pairs of $\<t^{b_1},\dots,t^{b_\ell}\>$
  are known, then we may use the same idea to compute the standard
  pairs of $\<t^{b_1},\dots,t^{b_\ell}, t^{b_{\ell+1}}\>$.
  
  Finally, finding the overlap classes and determining the maximal
  ones with respect to divisibility can be done by finding whether
  certain linear systems of equations and inequalities have integer solutions.
\end{proof}

\begin{remark}
  Having computed the (overlap classes of) standard pairs of $I$, the associated primes of
  $I$ and their corresponding multiplicities can be computed by inspection.
\end{remark}

\begin{example}[Continuation of Example~\ref{ex:ideals}\eqref{item:2dnormal}]
We return to $I= \langle x^2y^2, x^3y \rangle \subset
\field[x,xy,xy^2] \cong \field[\NN A]$ for $A=\big[\begin{smallmatrix}
  1 & 1 & 1 \\ 0 & 1 & 2 \end{smallmatrix}\big]$. In this case, we
illustrate how to compute the standard pairs of $I$ using the
method described in Theorem~\ref{thm:computeStdPairs}.

First we apply Theorem~\ref{thm:pairDifference} to $\NN A \minus
((2,2)+\NN A)$, to obtain a cover of standard monomials for
$I_{0}=  \langle x^2y^2 \rangle$. In this
case, the set~\eqref{eqn:upstairs} turns out to be 
\[
  \{(u,w) \in \NN^A \times \NN^{A} \mid b+A\cdot u = (2,2)^{t}+ A\cdot w \} = \{(u,w) \in \NN^A \times \NN^{A} \mid A(u-w)= (2,2)^{t} \}.
\]
A straightforward calculation shows that this set is the same as 
\[
\{(u,w): u-w = (0,2,0) \text{ or } (1,0,1) \}= \{ (u+w,w): u=(0,2,0) \text{ or }(1,0,1) , w \in \NN^{A}\}.
\]
It follows that the minimal solutions we are looking for are $(0,2,0)
\text{ and } (1,0,1)$. We see that the ideal $J \subset \field[\NN^A]
= \field[z_1,z_2,z_3]$ in the proof of
Theorem~\ref{thm:pairDifference} equals $\langle z_2^{2}, z_1z_3 \rangle$ and its
standard pairs are 
\[
\big((0,0,0),\{(0,0,1)\}\big), \; \big((0,1,0),\{(0,0,1)\}\big), \;
\big((0,0,0),\{(1,0,0)\}\big), \; \text{ and } \big((0,1,0),\{(1,0,0)\}\big).
\]
Thus, 
\[
\NN A \minus \bigg((2,2)^{t}+\NN A\bigg) = \NN \{(1,2) \} \cup \bigg( (1,1)+ \NN \{
(1,2)\}\bigg) \cup \NN \{(1,0)\} \cup \bigg((1,1)+ \NN \{(1,0)\}\bigg).
\]
These form a cover of the standard monomials of $\langle
x^2y^2\rangle$, and it is easily checked that this is the cover by
standard pairs of the standard monomials of $\langle x^2y^2\rangle$.

Now, to find the standard pairs of $I$, we compute $b+ \NN F \minus
\big((3,1)+\NN A\big)$ for each $b+ \NN F$ where $(b,F) \in
\stdPairs(\langle x^2y^2\rangle)$. The sets $\NN \{(1,2)\}$,
$(1,1)+\NN\{(1,2)\},$ and $\NN \{(1,0\}$ have an empty intersection with
$(3,1)+\NN A$. Hence we only need to compute $\big((1,1)+\NN
\{(1,0)\})\big) \minus \big( (3,1)+\NN A\big)$.

We apply Theorem \ref{thm:pairDifference} to do this, yielding $J=
\langle z_1^{2} \rangle$ in $\field[z_1]$. This ideal has two standard pairs
$(0,\varnothing)$ and $(1,\varnothing)$. It follows that
\[
\big((1,1)+\NN \{(1,0)\}\big) \minus \big( (3,1)+\NN A \big)= \{(0,0),
(1,1)\}
\]
The following is a cover of standard monomials of $I$.
\[
\{  \big((0,0), \{(1,2) \}\big) , \; \big( (1,1), \{ (1,2)\}\big), \;
\big((0,0), \{(1,0)\}\big) ,\; \big((0,0),\varnothing\big),
\big((1,1), \varnothing \big)\}
\]
We next use Proposition~\ref{prop:refineCover} to remove
$((0,0),\varnothing)$. Finally, the set of standard pairs of $I$ is 
\[
 \big((0,0), \{(1,2) \}\big), \;  \big( (1,1), \{ (1,2)\}\big), \; \big((0,0), \{(1,0)\}\big),\;  \text{ and } \big((1,1), \varnothing\big),
\]
as depicted in Figure \ref{fig:2dnormal_std}.
\end{example}

We now know how to compute the standard pairs of a monomial ideal $I$
in $\field[\NN A]$. It is natural to try to reverse the process and
find generators of a monomial ideal whose standard pairs are given.

\begin{theorem}
  \label{thm:pairsToGens}
The algorithm below has the set $C$ of standard pairs
  of a monomial ideal in $\field[\NN A]$ as its input. Its output is the set of
  generators for this ideal.
\begin{enumerate}[leftmargin=*]
\item For each $(a,F) \in C$
\begin{enumerate}[leftmargin=*]
\item Compute the decomposition $D$ of $a+\NN F \minus \big( \cup_{(b,G) \in \stdPairs(I) \mid 
    G \neq F} b+\NN G \big)$ using Theorem~\ref{thm:pairDifference}.
\item Set $J=\emptyset$.
\item For each $(\alpha,F') \in D$,
\begin{enumerate}[leftmargin=*]
\item Append $\alpha +a_{i}$ for all $a_{i} \in A \minus F$ to $J$
\end{enumerate}
\end{enumerate}
\item Let $J_1$ be the ideal generated by $J$; find its standard pairs of the ideal using Theorem~\ref{thm:computeStdPairs}.
\item If the set of standard pairs of $J_{1}$ coincides with $C$, return $J_{1}$. 
\item Otherwise, find standard pairs of $J_{1}$ that are not standard pairs in $C$. 
\item For such a standard pair $(a,F)$,
\begin{enumerate}
\item Find an element $b \in (a+\NN F) \cap I$ and add it to $J_{1}$.
\end{enumerate}
\item Go to (3).
\end{enumerate}
\end{theorem}

\begin{proof}
  If the standard pairs of a monomial ideal $I$ are known, we can
  determine which overlap classes are maximal with respect to
  divisibility (among all pairs belonging to the same face). For each
  standard pair $(a,F)$ in such an overlap class, we can compute
  $a+\NN F \minus \big( \cup_{(b,G) \in \stdPairs(I) \mid 
    G \neq F} b+\NN G \big)$ as a union over pairs $(\alpha,F')$ of
  sets $a+\NN F'$. For each such pair $(\alpha,F')$ and each $a_i$ mid $a_i \in A \minus F$, check whether $t^{a_i}t^{\alpha} \in I$ using the given standard pair.
  Let $J_1$ be the ideal generated
  by all monomials $t^{a_i}t^{\alpha}$ obtained in this way not in any standard pair. Then $J_1 \subset I$.

  Compute the standard pairs of $J_1$. If they coincide with the
  standard pairs of $I$, then $J_1=I$ and we are done.

  Otherwise, pick maximal overlap classes of standard pairs of $J_1$
  that are not standard pairs of $I$, remove all standard pairs of
  $I$, and use this to find elements of $I$ that do not belong to
  $J_1$. Obtain an ideal $J_2 \supsetneq J_1$.

  Repeat this procedure. Since $\field[\NN A]$ is Noetherian, this
  process must arrive at $I$ in a finite number of steps.
\end{proof}

\begin{remark}
  \label{rmk:computeIntersections}
  We observe that standard pairs can be used to compute intersections
  of monomial ideals. If $I$ and $J$ are monomial ideals in
  $\field[\NN A]$, then the union of the collections of standard pairs
  of $I$ and $J$ is a cover for the standard monomials of $I\cap
  J$. Applying Proposition~\ref{prop:refineCover} yields the standard pairs
  of $I\cap J$, and we can compute generators using
  Theorem~\ref{thm:pairsToGens}.
\end{remark}

\begin{example}
  We return to Example~\ref{ex:2dirred_decomposition_unique}~(ii). In this case, the standard pairs are $\big((0,0), \{(2,0)\}\big)$,
  $\big((0,1),\{(2,0)\}\big),$ and $\big((1,1),\{(2,0)\}\big)$; we wish
  to recover the generators of the ideal from this information, using
  the method from Theorem~\ref{thm:pairsToGens}. If we start with the
  standard pair $\big((0,1),\{(2,0)\}\big)$, we obtain the ideal 
  $J_{1} = \langle xy^{2},x^{2}y^{2} \rangle$. Using
  $\big((0,0),\{(2,0) \}\big)$ next, we find the
  monomials $y^{2},xy^{2}$, which generate $I$. 
\end{example}

  We can now compute irreducible decompositions of monomial ideals in
  $\field[\NN A]$.

\begin{theorem}
  \label{thm:computeIrreducibleDecomposition}
The algorithm below has the set $C$ of standard pairs
  of a monomial ideal $I$ in $\field[\NN A]$ as its input. Its output is the set of
  irreducible decomposition for $I$.
\begin{enumerate}[leftmargin=*]
\item Find an overlap class in $C$ which is maximal with respect to divisibility.
\item For such maximal overlap classes $[\bar{a},F]$,
\begin{enumerate}[leftmargin=*]
\item Find $U=\bigcup_{(a,F) \in [\bar{a},F]} \{ (u,v,w) \in \NN^A \times \NN^A
    \times \NN^F \mid A\cdot u + A\cdot v = a + F
    \cdot w \}$.
\item Project the set above as $\pi(U):=\{ u \in \NN^A : (u,v,w) \in U\}$. 
\item For each $u \in \pi(U)$, 
\begin{enumerate} 
\item Append all components of the decomposition $\NN A \minus u+\NN A$ by Theorem~\ref{thm:pairDifference} to the collection $C$.
\end{enumerate}
\item Generate the ideal $\pf$ whose standard cover is $C$ by Theorem~\ref{thm:pairsToGens}.
\end{enumerate}
\end{enumerate}
\end{theorem}

\begin{proof}
  Given the standard pairs of $I$, we can determine the associated
  primes of $I$. If $\pf$ is associated to $I$, let $[\bar{a},F]$ be an
  overlap class of standard pairs of $I$ that is maximal with respect
  to divisibility (among overlap classes belonging to $F$).

  Let $(a,F)$ be a standard pair of $I$ whose overlap class is
  $[\bar{a},F]$. Define
  \begin{equation}
    \label{eqn:divisors}
    \bigcup_{(a,F) \in [\bar{a},F]} \{ (u,v,w) \in \NN^A \times \NN^A
    \times \NN^F \mid A\cdot u + A\cdot v = a + F
    \cdot w \}.
  \end{equation}
We see that $u$ belongs to the projection of~\eqref{eqn:divisors} onto the
  first factor if and only if $A\cdot u$ divides an element of $a+\NN
  F$ where $(a,F) \in [\bar{a},F]$. Using
  Theorem~\ref{thm:primaryDecomposition} and
  Proposition~\ref{prop:irreducibleDecomposition}, 
  we see that these are (exponents of) the standard monomials in a
  valid irreducible component of $I$. 

  Adapting the method from Theorem~\ref{thm:pairDifference},
  we can find a cover of the standard monomials of this irreducible component.  
  Proposition~\ref{prop:refineCover} yields the corresponding standard
  pairs, and Theorem~\ref{thm:pairsToGens} provides generators.
\end{proof}

\begin{remark}
\label{rmk:computePrimaryDecomposition}
We can adapt the proof of Theorem~\ref{thm:computeIrreducibleDecomposition} to compute primary
components. Alternatively, we can compute the irreducible components
first, and then use Remark~\ref{rmk:computeIntersections} to intersect
all irreducible components associated to the same prime, yielding the corresponding primary component.
\end{remark}

\section{Implementation}
\label{sec:implementation}

In this section, we briefly discuss ongoing work towards implementation of the
algorithms presented in Section~\ref{sec:algorithms} in the computer
algebra system \texttt{SageMath}~\cite{SAGE} and \texttt{Macaulay2}~\cite{M2}. 

All algorithms in Section~\ref{sec:algorithms} were implemented in the \texttt{SageMath} called \texttt{StdPairs}~\cite{Yu22}. The three key ingredients in our algorithms are computing the face lattice of a cone, solving linear systems of equations and inequalities over $\ZZ$, and finding standard monomials of standard pairs in polynomial rings. The first can be done in \texttt{StdPairs} using the \texttt{Polyhedra} module of \texttt{SageMath}. \texttt{StdPairs} calculates all solution sets of integer linear equations for algorithms in Section~\ref{sec:algorithms} using a command \texttt{zsolve} from the software \texttt{4ti2}~\cite{4ti2}. Lastly, \texttt{StdPairs} uses the command \texttt{standardPairs} in \texttt{Macaulay2} (for details see~\cite{MR1949544}) to compute standard pairs over polynomial rings. All computational results of \texttt{StdPairs} can be archived as a binary file so that users may avoid repeated calculations.

Since \texttt{Macaulay2} also has access to the software \texttt{Polyhedra}~\cites{PolyhedraSource, PolyhedraArticle} and \texttt{4ti2}, the \texttt{StdPairs} package can be translated into a package for \texttt{Macaulay2}. The author of the \texttt{SageMath} version \texttt{StdPairs} package is currently translating \texttt{StdPairs} as a \texttt{C++} library. This will provide a \texttt{Macaulay2} interface for this package similar to the one for \texttt{4ti2} and we hope it will also increase the speed of computation. 

\bibliographystyle{plain}
\bibliography{stdPairs}
\end{document}